\documentclass[11pt]{article} 

\usepackage{amssymb}
\usepackage{mathrsfs}

\usepackage{mathtools}
 
\usepackage{latexsym}\usepackage{amsmath}\usepackage{xspace}\usepackage{enumerate}\usepackage{amsfonts}\usepackage{amscd}\usepackage{syntonly}\usepackage{amssymb}
\usepackage[alwaysadjust]{paralist}

\usepackage{epsfig} 
  
\newtheorem{thm}{Theorem}[section]
\newtheorem{lem}[thm]{Lemma}
\newtheorem{prop}[thm]{Proposition}

\newtheorem{defn} [thm]{Definition}
\newtheorem{rem} [thm]{Remark}

\newenvironment{proof}[1][Proof]{\noindent \textbf{#1:} }{\hspace{\stretch{1}} $\Box$\\}
 
\newtheorem{step}{Step}

\usepackage{bbm}
\usepackage{t1enc}

\newcommand{\eps}{\varepsilon}
\newcommand{\ph}{\varphi}

\newcommand\crit{\operatorname{crit}}

\newcommand\supp{\operatorname{supp}}

\def\ad{\operatorname{ad}}
\def\Ad{\operatorname{Ad}}
\def\coker{\operatorname{coker}}
\def\d{\operatorname{d}}

\def\dist{\operatorname{dist}}
\def\dom{\operatorname{dom}}
\def\dvol{\operatorname{dvol}}

\def\im{\operatorname{im}}
\def\ind{\operatorname{ind}}
\def\Ind{\operatorname{Ind}}

\def\pr{\operatorname{pr}}
\def\ran{\operatorname{ran}}
\def\reg{\operatorname{reg}}
\def\Stab{\operatorname{Stab}}
\def\supp{\operatorname{supp}}
\def\univ{\operatorname{univ}}

\def\A{\mathcal A}
\def\AA{\mathbbm A}
\def\B{\mathcal B}
\def\E{\mathcal E}
\def\F{\mathcal F}
\def\G{\mathcal G}
\def\H{\mathcal H}
\def\L{\mathcal L}
\def\O{\mathcal O}

\def\S{\mathcal S}
\def\V{\mathcal V}
\def\Z{\mathcal Z}
\def\YM{\mathcal{YM}}
\def\YMV{\mathcal{YM}^{\V}}

\def\R{\mathbb{R}}

 \newcommand\codim{\operatorname{codim}}
 
 \newcommand\di{\operatorname{d}\!}

\newcommand\hybr{\operatorname{hybr}}

\newcommand\CR{\mathcal{CR}}

\newcommand\EV{\mathcal{E}^{\V}}

\newcommand\He{\mathcal H}

\begin{document}


\title{The Yang--Mills Gradient Flow and Loop Spaces of Compact Lie Groups}


\author{Jan Swoboda}

\maketitle

\begin{abstract}
We study the $L^2$ gradient flow of the Yang--Mills functional on the space of connection $1$-forms on a principal $G$-bundle over the sphere $S^2$ from the perspective of Morse theory. The resulting Morse homology is compared to the heat flow homology of the space $\Omega G$ of based loops in the compact Lie group $G$. An isomorphism between these two Morse homologies is obtained by coupling a perturbed version of the Yang--Mills gradient flow with the $L^2$ gradient flow of the classical action functional on loops. Our result gives a positive answer to a question due to Atiyah.
\end{abstract}


\tableofcontents

 

%
%

\section{Introduction}

Let $\Sigma:=S^2$ denote the unit sphere in the euclidian space $\mathbbm R^3$. Let $G$ be a compact Lie group, $\mathfrak g$ its Lie algebra (endowed with an $\Ad$-invariant inner product), and $P$ a principal $G$-bundle over $\Sigma$. In this paper we give an affirmative answer to a question raised by Atiyah relating Yang--Mills Morse homology of the space of gauge equivalence classes of $\mathfrak g$-valued connections on $P$ to heat flow homology of the group $\Omega G$ of based loops in $G$. The Morse complexes we shall be concerned with are the complex generated by the $L^2$ gradient flow of the Yang--Mills functional $\YM$  on the one hand, and the complex generated by the $L^2$ gradient flow of the classical action functional on $\Omega G$ on the other. Our goal is to establish a chain isomorphism between these two Morse complexes.\\
\noindent\\
Let us describe our setup. We denote by $\ad(P):=P\times_{\Ad}\mathfrak g$ the adjoint Lie algebra bundle over $\Sigma$, and by $\A(P)$ the space of $\mathfrak g$-valued $1$-forms on $P$. The latter is an affine space over $\Omega^1(\Sigma,\ad(P))$, the space of $\ad(P)$-valued $1$-forms on $\Sigma$. The curvature of a connection $A\in\A(P)$ is $F_A=dA+\frac{1}{2}[A\wedge A]\in\Omega^2(\Sigma,\ad(P))$. The space $\A(P)$ is acted on by the groups $\mathcal G(P)$ and $\G_0(P)$ of gauge, respectively based gauge transformations of $P$, cf.~Section \ref{sec:prelim} for precise definitions. On $\A(P)$ we consider the  $\mathcal G(P)$-invariant Yang--Mills functional 
\begin{eqnarray}\label{YMintroduction}
\mathcal{YM}\colon\A(P)\to\mathbbm R,\qquad\mathcal{YM}(A)=\frac{1}{2}\int_{\Sigma}\big\langle F_A\wedge\ast F_A\big\rangle.
\end{eqnarray}

The corresponding Euler-Lagrange equation is the second order partial differential equation $d_A^{\ast}F_A=0$, called \emph{Yang--Mills equation}. Critical points of $\YM$ are degenerate (due to the gauge invariance of the functional) but satisfy the so-called Morse--Bott condition, cf.~the discussion in Section \ref{sect:critmanifolds}. We shall be concerned with a perturbed version of the negative $L^2$ gradient flow equation associated with \eqref{YMintroduction}, which is the PDE
\begin{eqnarray}\label{introdYMgrad}
\partial_sA+d_A^{\ast}F_A-d_A\Psi+\nabla\V^-(A)=0.
\end{eqnarray}

Cf.~Section \ref{sect:perturbations} below for the precise form of the perturbation $\V^-\colon\A(P)\to\mathbbm R$. The term $d_A\Psi$ (where $\Psi\in\Omega^0(\Sigma,\ad(P))$) is introduced to make \eqref{introdYMgrad} invariant under time-dependent gauge transformations.\\
Let $S^1:=\mathbbm R/2\pi\mathbbm Z$. The {\emph{free loop group}} of $G$ is the space $\Lambda G:=C^{\infty}(S^1,G)$, endowed with the group multiplication defined by $(x_1x_2)(t):=x_1(t)x_2(t)$ for $x_1,x_2\in\Lambda G$. The {\emph{based loop group}} of $G$ is the subgroup 
\begin{eqnarray*}
\Omega G:=\left\{x\in\Lambda G\mid x(0)=\mathbbm 1\right\} 
\end{eqnarray*}
of $\Lambda G$. Throughout we will identify $\Omega G$ with the quotient of $\Lambda G$ modulo the free action of $G$ defined by $(h\cdot x)(t):=hx(t)$ for $h\in G$ and $x\in\Lambda G$. We endow $G$ with the biinvariant Riemannian metric induced by the $\Ad$-invariant inner product on $\mathfrak g$ and consider on $\Lambda G$ the classical action functional
\begin{eqnarray}\label{Eintroduction}
\mathcal E\colon\Lambda G\to\mathbbm R,\qquad\mathcal E(x)=\frac{1}{2}\int_0^{2\pi}\|\partial_tx(t)\|^2\,dt.
\end{eqnarray}

It descends to a functional on $\Omega G$ by biinvariance of the Riemannian metric on $G$. It is well-known that the critical points of $\mathcal E$ are precisely the closed geodesics in $G$. As a consequence of the invariance of the functional $\E$ under conjugation with elements $h\in G$, it follows that critical points of $\E$ are degenerate. However, also here it turns out that the Morse--Bott condition is satisfied. The (perturbed) negative $L^2$ gradient flow equation resulting from \eqref{Eintroduction} is the PDE
\begin{eqnarray}\label{introdloopgradient1}
\partial_sx-\nabla_t\partial_tx+\nabla\V^+(x)=0.
\end{eqnarray}

For the precise form of the perturbation $\V^+\colon\Omega G\to\mathbbm R$, we refer to Section \ref{sect:perturbations} below.\\
\noindent\\
Morse homology groups for loop spaces of compact Lie groups and homogeneous spaces have been computed in a classical paper by Bott \cite{Bott}, which constitutes an application of Morse theoretic ideas in the context of infinite dimensional Hilbert manifolds. For further applications to the theory of closed geodesics on general compact manifolds we refer to Klingenberg \cite{Klingenberg}. However, in both instances, Morse theory is based on a $W^{1,2}$ gradient flow, leading to an ODE in Hilbert space. In contrast, the $L^2$ gradient flow approach to Morse theory on loop spaces of compact Riemannian manifolds has only recently been investigated by Weber \cite{Web} and uses techniques from parabolic PDEs (cf.~also Salamon and Weber \cite{SalWeb} for an application to Floer homology of cotangent bundles). In the present work, we shall follow the latter approach and specialize some of the results in \cite{Web} to loop spaces of compact Lie groups. Complementary to heat flow homology, a Morse homology theory based on the $L^2$ gradient flow \eqref{introdYMgrad} on compact Riemann surfaces of arbitrary genus has been laid down by this author in \cite{Swoboda1}. These results are used throughout the present article.

\subsubsection*{Main results}
In their seminal paper \cite{AB}, Atiyah and Bott studied the Yang--Mills functional $\mathcal{YM}$ over a compact Riemann surface from a Morse--Bott theoretical point of view. This led them to the discovery of a close correspondence between the Morse theoretical picture of a stratification of the space $\A(P)/\mathcal G(P)$ into stable manifolds and certain moduli spaces of semi-stable holomorphic vector bundles, and initiated a lot of further research in algebraic geometry as documented e.g.~by the review article \cite{Kirwan} by Kirwan. In \cite{AB}, Atiyah and Bott pointed out that in the genus zero case the Yang--Mills critical points correspond via a so called holonomy map to closed geodesics in $G$. This observation was subsequently made more explicit through works by Gravesen \cite{Gravesen} and Friedrich and Habermann \cite{FriedrichHabermann}. In these articles, a holonomy map $\Phi\colon\A(P)\to\Omega G$ is constructed which assigns to a connection $A$ its holonomy along the greater arcs  in $\Sigma$ connecting the north and south pole, cf.~Appendix \ref{holmapappendix} for details. The map $\Phi$ is equivariant with respect to the actions of $\mathcal G(P)$ by gauge transformations and of $G$ by conjugation. It furthermore maps critical points of the Yang--Mills functional to closed geodesics in $G$ (of a certain homotopy type, determined by the bundle $P$), preserving the Morse indices. The natural question, raised by Atiyah, whether this apparent close relation between the aforementioned sets of generators of Morse complexes extends to the full Morse theory picture has not been resolved so far. However, a formal consideration invoking an adiabatic limit of a certain deformation of the standard Riemannian metric on $\Sigma=S^2$ indicates a positive answer to his question. Namely it is suggested that for a family of Riemannian metrics on $\Sigma$ which approximates a $\delta$-impulse on the equator, connecting trajectories of the Yang--Mills gradient flow could be constructed in a bijective way from from those of the heat flow \eqref{introdloopgradient1}. Such an approach, although successfully been followed in many related contexts (cf.~e.g.~\cite{DoSal,JaNo}), did not provide an answer to that question. However, Davies has obtained some interesting preliminary results in his unpublished PhD thesis \cite{Davies} supervised by Salamon.\\
The aim of the present paper is to settle Atiyah's question following a completely different approach. The guiding idea in our proof is to combine the $L^2$ gradient flows \eqref{introdYMgrad} and \eqref{introdloopgradient1} by studying a so-called {\emph{hybrid moduli space problem}}. For a given pair $\hat C^{\pm}$ of critical manifolds of the functionals $\YM$, respectively $\E$, we shall consider the space of configurations
\begin{multline*}
\lefteqn{\hat{\mathcal M}(\hat{\mathcal C}^-,\hat{\mathcal C}^+):=}\\
\big\{(A,\Psi,x)\in C^{\infty}(\mathbbm R^-,\A(P)\times\Omega^0(\Sigma,\ad(P)))\times\,C^{\infty}(\mathbbm R^+,\Lambda G)\;\big|\\
(A,\Psi)\;\textrm{satisfies}\;\eqref{introdYMgrad},\quad x\;\textrm{satisfies}\;\eqref{introdloopgradient1},\quad x(0)=h\Phi(A(0))\;\textrm{for some}\;h\in G,\\
\lim_{s\to-\infty}(A(s),\Psi(s))=\big(A^-,0)\in\hat{\mathcal C}^-\times\Omega^0(\Sigma,\ad(P)),\lim_{s\to+\infty}x(s)=x^+\in\hat{\mathcal C}^+\big\}.
\end{multline*}

Hence $\hat{\mathcal M}(\hat{\mathcal C}^-,\hat{\mathcal C}^+)$ is the moduli space of tuples $(A,\Psi,x)$ such that $(A,\Psi)$ solves the perturbed Yang--Mills gradient flow equation \eqref{introdYMgrad} on the negative time interval $(-\infty,0]$, while $x$ is a solution of the perturbed loop group gradient flow equation \eqref{introdloopgradient1} on the positive time interval $[0,\infty)$. Both solutions are coupled under the holonomy map $\Phi\colon\A(P)\to\Omega G$ as introduced above. The moduli space $\mathcal M(\mathcal C^-,\mathcal C^+)$ then to be studied is the quotient of $\hat{\mathcal M}(\hat{\mathcal C}^-,\hat{\mathcal C}^+)$ modulo the actions by gauge transformations and left translations $x\mapsto hx$ (for $h\in G$).\\
As pointed out before, the sets of critical points of both functionals $\YM$ and $\E$ are degenerate in a Morse--Bott sense. This fact requires us to use a certain variant of Morse theory, called \emph{Morse theory with cascades}, as introduced by Frauenfelder in \cite{Frauenfelder1} and described in \cite[Section 8.1]{Swoboda1}. Throughout we shall work on fixed sublevel sets $\{A\in\A(P)\mid\YM(A)\leq a\}$ and $\{x\in\Lambda G/G\mid\E(x)\leq b\}$ (where usually $b={4a}/\pi$). As an additional datum, we fix a Morse function $h$ on the union of critical manifolds of $\YM$ below the level $a$ (respectively of $\E$ below the level $b$), the discrete set of critical points of which are the generators of two Morse complexes 
\begin{eqnarray}\label{introdcomplexes}
CM_{\ast}^a\big(\A(P)/\mathcal G_0(P),\V^-,h\big)\qquad\textrm{and}\qquad CM_{\ast}^b\big(\Lambda G/G,\V^+,h\big).
\end{eqnarray}

Here and throughout this article we adopt the convention that  $\Omega G=\Lambda G/G$ shall denote the connected component of the based loop group which contains the image of $\A(P)$ under the map $\Phi$, cf.~Appendix \ref{holmapappendix}. It is determined by the equivalence class of the principal $G$-bundle $P$. Our goal is to set up a chain homomorphism $\Theta$ between the complexes in \eqref{introdcomplexes}. It is defined for a pair of generators of equal Morse index by a count of elements in a certain \emph{moduli space with cascades}. These are constructed from a suitable concatenation of elements in $\mathcal M(\mathcal C^-,\mathcal C^+)$ and negative gradient flow lines of the function $h$, cf.~ Definition \ref{hybridflowline}. The key observation, allowing us to show invertibility of the homomorphism $\Theta$, is the property of the holonomy map $\Phi\colon\A(P)\to\Omega G$ to decrease energy. Namely, for any connection $A\in\A(P)$ there holds the inequality
\begin{eqnarray}\label{introenergy}
\mathcal{YM}^{\V^-}(A)\geq\frac{\pi}{4}\mathcal E^{\V^+}(\Phi(A)),
\end{eqnarray} 
(cf.~Lemma \ref{decomplemma}), with equality if $A$ is a Yang--Mills connection. Inequality \eqref{introenergy} is not a new result and can be found in Gravesen \cite{Gravesen}. In our context it leads directly to the proof of invertibility of $\Theta$ and thus implies the desired isomorphism in Morse homology.

\begin{thm}[Main result]\label{thm:mainresult} 
Let $G$ be a compact Lie group, and $P$ a principal $G$-bundle $P$ over $\Sigma$. Let $a\geq0$ be a regular value of $\YM$ and set $b:=4a/\pi$. Then, for a generic $a$-admissible perturbation $\V=(\V^-,\V^+)\in Y_a$ (cf.~Definition \ref{def:regperturbation} below) the chain homomorphism
\begin{eqnarray*}
\Theta_{\ast}\colon CM_{\ast}^a\big(\A(P)/\mathcal G_0(P),\V^-,h\big)\to CM_{\ast}^b\big(\Lambda G/G,\V^+,h\big)
\end{eqnarray*}
induces an isomorphism
\begin{eqnarray*}
\fbox{$\displaystyle [\Theta_{\ast}]\colon HM_{\ast}^a\big(\A(P)/\mathcal G_0(P),\V^-,h\big)\to HM_{\ast}^b\big(\Lambda G/G,\V^+,h\big)$}
\end{eqnarray*}
of Morse homology groups.
\end{thm}

Let us point out here that the method of defining chain maps between Morse homology complexes by utilizing hybrid moduli spaces is a fairly recent one. It has successfully been employed by Abbondandolo and Schwarz \cite{AbSchwarz} in proving that Floer homology $HF_{\ast}(T^{\ast}M)$ of cotangent bundles $T^{\ast}M$ is isomorphic to singular homology of the free loop space $\Lambda M$ ($M$ a compact manifold). In their situation, a crucial role is played by an inequality similar to \eqref{introenergy}, relating the symplectic action to the classical action via Legendre duality. 

\subsubsection*{Further directions}
\paragraph{$G$-equivariant Morse homology} The group $G$ acts in a natural way on the quotient spaces $\G(P)/G_0(P)$ and $\Lambda G/G$. In the first case,  it is given by $g\cdot[A]=[g^{\ast}A]$ for $g\in G\cong\G(P)/\G_0(P)$. In the second case the group $G$ acts by conjugation $g[x]=[g^{-1}xg]$. In his thesis \cite{Swoboda} the author has worked out a $G$-equivariant version of Theorem \ref{thm:mainresult}. This is mainly a technical extension and requires to replace the spaces $\A(P)/\G_0(P)$ and $\Lambda G/G$ by $(\A(P)\times E_nG)/\G(P)$, respectively by $(\Lambda G\times E_nG)/(G\times G)$, for a suitable finite-dimensional approximation $E_nG$ of the classifying space $EG$.

\paragraph{Higher genus surfaces} Morse homology $HM_{\ast}^a\big(\A(P)/\mathcal G_0(P),\V^-,h\big)$ as considered in this article for $\Sigma=S^2$ has more generally been defined in \cite{Swoboda1} for closed Riemann surfaces of arbitrary genus. It is known from \cite{Davies,Gravesen} that Yang--Mills connections on principal $G$-bundles $P$ over such surfaces correspond bijectively to certain geodesic polygons in the Lie group $G$. Moreover, an estimate similar to \eqref{decompequation} relating the energy functionals $\YM$ and $\E$ continuous to hold true in this more general situation. Hence one should be able to prove a version of Theorem \ref{thm:mainresult} for higher genus surfaces, but this is open at present.

\subsubsection*{Acknowledgements}
This work is based on the author's PhD thesis \cite{Swoboda}. He would like to express his gratitude to his advisor D.~A.~Salamon for his support while working on this project. He would like to thank A.~Oancea for bringing to his attention the work \cite{AbSchwarz}. A discussion with M.~Atiyah concerning some of the background and history of the question treated in this article is greatfully acknowledged. Many thanks also to  W.~Ballmann, R.~Janner, M.~Schwarz, M.~Struwe, and J.~Weber for fruitful discussions.

\section{Critical manifolds, Yang--Mills gradient flow lines, and Morse complexes}

\subsection{Preliminaries}\label{sec:prelim}
Let $\Sigma:=S^2$ be the unit sphere in $\mathbbm R^3$, endowed with the standard round metric. Let $G$ be a compact Lie group with Lie algebra $\mathfrak g$. On $\mathfrak g$ we fix an $\Ad$-invariant inner product $\langle\,\cdot\,,\,\cdot\,\rangle$, which exists by compactness of $G$. Let $P$ be a principal $G$-bundle over $\Sigma$. A gauge transformation is a section of the bundle $\Ad(P)\coloneqq P\times_GG$ associated to $P$ via the action of $G$ on itself by conjugation $(g,h)\mapsto g^{-1}hg$. Let $\ad(P)$ denote the Lie algebra bundle associated to $P$ via the adjoint action
\begin{eqnarray*}
(g,\xi)\mapsto\left.\frac{d}{dt}\right|_{t=0}g^{-1}\exp(t\xi)g\qquad(\textrm{for}\,g\in G,\,\xi\in\mathfrak g)
\end{eqnarray*}
of $G$ on $\mathfrak g$. We denote the space of smooth $\ad(P)$-valued differential $k$-forms by $\Omega^k(\Sigma,\ad(P))$, and by $\A(P)$ the space of smooth connections on $P$. The latter is an affine space over $\Omega^1(\Sigma,\ad(P))$. The group $\G(P)$ acts on $\A(P)$ by gauge transformations. We call a connection $A\in\A(P)$ \emph{irreducible} if the stabilizer subgroup $\Stab A\subseteq\G(P)$ is trivial. Otherwise it is called \emph{reducible}. It is easy to show that $\Stab A$ is a compact Lie group, isomorphic to a subgroup of $G$. Let $z\in\Sigma$ be arbitrary but fixed. We let $\G_0(P)\subseteq\G(P)$ denote the group of \emph{based gauge transformation}, i.e.~those gauge transformations which leave the fibre $P_z\subseteq P$ above $z$ pointwise fixed. It is a well-known fact that $\G_0(P)$ acts freely on $\A(P)$.\\
\noindent\\
On $\A(P)$ we define a gauge-invariant $L^2$ inner product by $\langle\alpha,\beta\rangle=\int_{\Sigma}\langle\alpha\wedge\ast\beta\rangle$ for $\alpha,\beta\in\Omega^1(\Sigma,\ad(P))$. The curvature of the connection $A$ is the $\ad(P)$-valued $2$-form $F_A=dA+\frac{1}{2}[A\wedge A]$. It satisfies the Bianchi identity $d_AF_A=0$. The covariant exterior differential induced by $A\in\A(P)$ is the operator
\begin{eqnarray*}
d_A\colon\Omega^k(\Sigma,\ad(P))\to\Omega^{k+1}(\Sigma,\ad(P)),\quad\alpha\mapsto d\alpha+[A\wedge\alpha].
\end{eqnarray*}
The formal adjoint $d_A^{\ast}$ of it is given by $d_A=-\ast d_A\ast$. The covariant Hodge Laplacian on forms is the operator $\Delta_A\coloneqq d_A^{\ast}d_A+d_Ad_A^{\ast}$. The \emph{perturbed Yang--Mills functional} $\YMV$ has been introduced in \eqref{YMintroduction}. If $\V=0$, we write $\YM$ and call this the \emph{unperturbed Yang--Mills functional}. The $L^2$ gradient of $\YMV$ at $A\in\A(P)$ is $\nabla\YMV(A)=d_A^{\ast}F_A+\nabla\V(A)\in\Omega^1(\Sigma,\ad(P))$. Its Hessian is the second order differential operator
\begin{eqnarray*} 
H_A\YMV=d_A^{\ast}d_A+\ast[\ast F_A\wedge\,\cdot\,]+H_A\V\colon\Omega^1(\Sigma,\ad(P))\to\Omega^1(\Sigma,\ad(P)).
\end{eqnarray*}
We also make use of the notation $H_A\coloneqq d_A^{\ast}d_A+\ast[\ast F_A\wedge\,\cdot\,]$.\\
\noindent\\
Throughout we will use Sobolev spaces of sections of vector bundles and Banach manifolds modeled on such Sobolev spaces, like e.g.~various groups of gauge transformations. A detailed account of this subject is given in the book \cite[Appendix B]{Wehrheim}. We therefore keep the discussion of these matters short. Let $1\leq p\leq\infty$ and $k\geq0$ an integer. We fix a smooth reference connection $A\in\A(P)$. It determines a covariant derivative $\nabla_A$ on $\Omega^{\ast}(\Sigma,\ad(P))$. We employ the notation $W^{k,p}(\Sigma,\ad(P))$ and $W^{k,p}(\Sigma,T^{\ast}\Sigma\otimes\ad(P))$ for the Sobolev spaces of $\ad(P)$-valued $0$- and $1$-forms whose weak derivatives (with respect to $\nabla_A$) up to order $k$ are in $L^p$. These spaces are independent of the choice of $A$. However, for $k\geq1$, the corresponding norms depend on this choice.  The standard Sobolev embedding and Rellich--Kondrachov compactness theorems apply to these spaces. The affine $(k,p)$-Sobolev space of connections on $P$ is defined as
\begin{eqnarray*}
\A^{k,p}(P)\coloneqq A+W^{k,p}(\Sigma,T^{\ast}\Sigma\otimes\ad(P)).
\end{eqnarray*}
This definition is again independent of the choice of smooth reference connection $A$. To define Sobolev spaces of gauge transformations we need to assume $kp>\dim\Sigma=2$. Then let $\G^{k,p}(P)$ denote the set of equivariant maps $P\to G$ which are of the form $g=g_0\exp(\ph)$, where $g_0$ is a smooth such map and $\ph\in W^{k,p}(\Sigma,\ad(P))$. (Here we view $\ph$ as an equivariant map $P\to\mathfrak g$). The space $\G^{k,p}(P)$ is a Banach manifold modeled on $W^{k,p}(\Sigma,\ad(P))$. As a well-known fact we remark that $\G^{k,p}(P)$ is a group with smooth group multiplication and inversion. It acts smoothly on $\A^{k-1,p}(P)$ by gauge transformations. Let $I$ be a finite or infinite interval. We often make use of the parabolic Sobolev space
\begin{eqnarray}\label{eq:defparabolicSobolev}
W^{1,2;p}(I\times\Sigma,\ad(P))\coloneqq L^p(I,W^{2,p}(\Sigma,\ad(P)))\cap W^{1,p}(I,L^p(\Sigma,\ad(P)))
\end{eqnarray} 
of $\ad(P)$-valued $0$-forms admitting one time and two space derivatives in $L^p$ (and similarly for $\ad(P)$-valued $1$-forms). The parabolic Sobolev space $\A^{1,2;p}(P)$ of connections is defined analogously, with $W^{2,p}(\Sigma,\ad(P))$ and $L^p(\Sigma,\ad(P))$ in \eqref{eq:defparabolicSobolev} replaced by $\A^{2,p}(P)$, respectively $\A^{0,p}(P)$. Note that when there is no danger of confusion, we for ease of notation write $L^p(\Sigma)$ instead of $L^p(\Sigma,\ad(P))$ or $L^p(\Sigma,T^{\ast}\Sigma\otimes\ad(P))$ (and similarly for the other Sobolev spaces). Further notation frequently used is $\dot A\coloneqq\partial_sA\coloneqq\frac{d A}{ds}$, etc.~for derivatives with respect to the time parameter $s$.

\subsection{Banach spaces of abstract perturbations}\label{sect:perturbations}

\subsubsection*{Perturbations of the Yang--Mills funcional}\label{Bspaceperturbations}

Our construction of a Banach space of perturbations is based on the following $L^2$ local slice theorem due to Mrowka and Wehrheim \cite{MrowkaWehrheim}. We fix $p>2$ and let
\begin{eqnarray*}
\S_{A_0}(\varepsilon)\coloneqq\big\{A=A_0+\alpha\in\A^{0,p}(P)\,\big|\,d_{A_0}^{\ast}\alpha=0,\|\alpha\|_{L^2(\Sigma)}<\varepsilon\big\}
\end{eqnarray*}
denote the set of $L^p$-connections in the local slice of radius $\varepsilon$ with respect to the reference connection $A_0\in\A^{0,p}(P)$.

\begin{thm}[$L^2$ local slice theorem]\label{locslicethm} 
Let $p>2$. For every $A_0\in\A^{0,p}(P)$ there are constants $\varepsilon,\delta>0$ such that the map
\begin{eqnarray*}
\mathfrak m\colon\big(\S_{A_0}(\varepsilon)\times\G^{1,p}(P)\big)/\Stab{A_0}\to\A^{0,p}(P),\qquad[(A_0+\alpha,g)]\mapsto(g^{-1})^{\ast}(A_0+\alpha)
\end{eqnarray*}
is a diffeomorphism onto its image, which contains an $L^2$ ball,
\begin{eqnarray*}
B_{\delta}(A_0)\coloneqq\big\{A\in\A^{0,p}(P)\,\big|\,\|A-A_0\|_{L^2(\Sigma)}<\delta\big\}\subseteq\im\mathfrak m.
\end{eqnarray*}
\end{thm}

\begin{proof}
For a proof we refer to \cite[Theorem 1.7]{MrowkaWehrheim}.
\end{proof}

We fix the following data.

\begin{compactenum}[(i)]
\item
A dense sequence $(A_i)_{i\in\mathbbm N}$ of irreducible smooth connections in $\A(P)$.
\item
For each $i\in\mathbbm N$ a dense sequence $(\eta_{ij})_{j\in\mathbbm N}$ of smooth $1$-forms in $\Omega^1(\Sigma,\ad(P))$ satisfying $d_{A_i}^{\ast}\eta_{ij}=0$ for all $j\in\mathbbm N$.
\item
A smooth cutoff function $\rho\colon\mathbbm R\to[0,1]$ such that $\rho=1$ on $[-1,1]$, $\supp\rho\subseteq[-4,4]$, and $\|\rho'\|_{L^{\infty}(\R)}<1$. Set $\rho_k(r)\coloneqq\rho(k^2r)$ for $k\in\mathbbm N$.
\end{compactenum}
We fix $i\in\mathbbm N$ and a constant $\eps_i>0$ such that the conclusion of Theorem \ref{locslicethm} applies for $A_0\coloneqq A_i$ and this constant $\eps_i$. Note that by assumption, $\Stab A_i=\mathbbm 1$. Theorem \ref{locslicethm} thus implies that the map
\begin{eqnarray*}
\mathfrak m_i\colon\mathcal S_{A_i}(\eps_i)\times\G^{1,p}(P)\to\A^{0,p}(P),\quad(A_i+\alpha,g)\mapsto(g^{-1})^{\ast}(A_i+\alpha)
\end{eqnarray*}
is a diffeomorphism onto its image. Hence 
\begin{eqnarray*} 
\alpha_i\colon\im\mathfrak m_i\to L^p(\Sigma,\ad(P)),\quad A\mapsto(\pr_1\circ\mathfrak m^{-1})(A)-A_i
\end{eqnarray*}
(with $\pr_1\colon\mathcal S_{A_i}(\eps_i)\times\G^{1,p}(P)\to\mathcal S_{A_i}(\eps_i)$ denoting projection) is a well-defined smooth map with image being contained in $\mathcal S_{A_i}(\eps_i)-A_i$. We extend $\alpha_i$ to a map on $\A(P)$ by setting $\alpha_i(A)=0$ for $A\in\A^{0,p}(P)\setminus\im\mathfrak  m_i$. Hence 
\begin{eqnarray}\label{modelperturbation}
\V_{\ell}^-\colon\A(P)\to\mathbbm R,\qquad A\mapsto\rho_k(\|\alpha_i(A)\|_{L^2(\Sigma)}^2)\langle\eta_{ij}+\alpha_i(A),\eta_{ij}\rangle
\end{eqnarray}
is a well-defined map for every triple $\ell=(i,j,k)\in\mathbbm N^3$. Note also that $\V_{\ell}^-$ is invariant under the action of $\G^{1,p}(P)$ by gauge transformations.

\begin{prop}\label{prop:condlocslicepert}
For every $A\in\im\mathfrak m_i$ there exists a unique $g\in\G^{1,p}(P)$ such that $g^{\ast}A-A_i=\alpha_i(A)$ and $d_{A_i}^{\ast}\alpha_i(A)=0$.
\end{prop}

\begin{proof}
For a proof we refer to \cite[Proposition 2.6]{Swoboda1}.
\end{proof}

\begin{prop}\label{prop:perturbsmooth}
Let $\ell=(i,j,k)\in\mathbbm N^3$ such that $k>\frac{10}{\delta_i}$. Then the map $\V_{\ell}^-\colon\A^{0,p}(P)\to\mathbbm R$ defined in \eqref{modelperturbation} is smooth.
\end{prop}

\begin{proof}
For a proof we refer to \cite[Proposition 2.7]{Swoboda1}.
\end{proof}

\begin{prop}\label{prop:L2estgrad}
Let $A_0\in\A(P)$ and $p>2$. There exist constants $c(A_0)$, $c(A_0,p)$ and $\delta(A_0,p)$ such that the estimates 
\begin{align*}
\|\alpha(A)\|_{W^{1,p}(\Sigma)}&\leq c(A_0,p)\big(1+\|F_A\|_{L^p(\Sigma)}\big),\\
\|\nabla\V^-(A)\|_{C^0(\Sigma)}&\leq c(A_0)\big(1+\|F_A\|_{L^3(\Sigma)}\big),\\
\|d_A\nabla\V^-(A)\|_{L^p(\Sigma)}&\leq c(A_0,p)\big(1+\|F_A\|_{L^p(\Sigma)}+\|\alpha(A)\|_{L^{2p}(\Sigma)}^2\big)
\end{align*}
are satisfied for all $A\in\A^{0,p}(P)$ with $\|\alpha(A)\|_{L^2(\Sigma)}<\delta(A_0,p)$. 
\end{prop}

\begin{proof}
For a proof we refer to \cite[Proposition A.4]{Swoboda1}.
\end{proof}

\begin{rem}\label{rem:perturbationsYMnonneg}
\upshape
In the following we shall admit only those perturbations $\V_{\ell}^-$ which meet the assumptions of Propositions \ref{prop:perturbsmooth} and \ref{prop:L2estgrad}. These are precisely satisfied for triples $\ell=(i,j,k)\in\mathbbm N^3$ such that $k>\max\{\frac{10}{\delta_i},\frac{2}{\delta(A_i,p)}\}$ where $\delta_i$ denotes the constant of Proposition \ref{prop:perturbsmooth}, and $\delta(A_i,p)$ is as in Proposition \ref{prop:L2estgrad}. Moreover, the Cauchy--Schwarz inequality implies that $\langle\eta_{ij}+\alpha(A),\eta_{ij}\rangle\geq0$ if $\|\alpha(A)\|_{L^2(\Sigma)}\leq\|\eta_j\|_{L^2(\Sigma)}$, and hence the map $\V_{\ell}^-$ is non-negative for sufficiently large indices $k$. For the remainder of this article we allow only for triples $(i,j,k)\in\mathbbm N^3$ such that $k$ satisfies these conditions and renumber the subset of such triples by integers $\ell\in\mathbbm N$. 
\end{rem}

Given $\ell\in\mathbbm N$, we fix a constant $C_{\ell}>0$ such that the following conditions are satisfied. 
 
\begin{compactenum}[(i)]
\item
$\sup_{A\in\A(P)}|\V_{\ell}^-(A)|\leq C_{\ell}$,
\item
$\sup_{A\in\A(P)}\|\nabla\V_{\ell}^-(A)\|_{L^2(\Sigma)}\leq C_{\ell}$,
\item
$\|\nabla\V_{\ell}^-(A)\|_{C^0(\Sigma)}\leq C_{\ell}(1+\|F_A\|_{L^3(\Sigma)})$ for all $A\in\A(P)$.
\item
$\|H_A\V_{\ell}^-\beta\|_{L^p(\Sigma)}\leq C_{\ell}(1+\|F_A\|_{L^3(\Sigma)})\|\beta\|_{L^p(\Sigma)}$ for all $A\in\A(P)$, $\beta\in\Omega^1(\Sigma,\ad(P))$, and $1<p<\infty$. 
\end{compactenum}
Here the notation $\nabla\V_{\ell}^-\colon\A(P)\to\Omega^1(\Sigma,\ad(P))$ and $H_A\V_{\ell}^-\colon\Omega^1(\Sigma,\ad(P))\to\Omega^1(\Sigma,\ad(P))$ refers to the $L^2$ gradient and Hessian (at the point $A\in\A(P)$) of the map $\V_{\ell}^-$. The existence of the constant $C_{\ell}$ has been shown in \cite[Proposition A.6]{Swoboda1}. The \emph{universal space of perturbations} is the normed linear space
\begin{eqnarray*} 
Y\coloneqq\Big\{\V^-\coloneqq\sum_{\ell=1}^{\infty}\lambda_{\ell}\V_{\ell}^-\,\Big|\,\lambda_{\ell}\in\R\;\textrm{and}\;\|\V^-\|\coloneqq\sum_{\ell=1}^{\infty}C_{\ell}|\lambda_{\ell}|<\infty\Big\}.
\end{eqnarray*}
It is a separable Banach space isomorphic to the space $\ell^1$ of absolutely summable real sequences. Further relevant properties of the perturbations $\V_{\ell}^-$ are discussed in \cite[Appendix A]{Swoboda1}.

\subsubsection*{Perturbations of the loop group energy functional}

We shall follow closely Salamon and Weber \cite{SalWeb} in our construction of a Banach space $Y^+$ of perturbations of the loop group energy functional $\E$. Let us fix the following data.

\begin{compactenum}[(i)]
\item
A dense sequence $(x_i)_{i\in\mathbbm N}$ of points in $\Omega G$.
\item 
For each $i\in\mathbbm N$ a dense sequence $(\eta_{ij})_{j\in\mathbbm N}$ in $T_{x_i}(\Omega G)$.
\item 
A smooth cut-off function $\rho\colon\mathbbm R\to[0,1]$ supported in $[-4,4]$, and satisfying $\rho=1$ on $[-1,1]$ and $\|\rho'\|_{L^{\infty}(\mathbbm R)}<1$. For $k\in\mathbbm N$, set $\rho_{k}(r):=\rho(k^2r)$.
\end{compactenum}

Denote by $\iota>0$ the injectivity radius of the compact Riemannian manifold $G$. Fix a further cut-off function $\beta$ supported in $[-\iota^2,\iota^2]$ such that $\beta=1$ on $\big[-\frac{\iota^2}{4},\frac{\iota^2}{4}\big]$. For $x_i\in\Omega G$ as in (i) above and $q\in G$ within distance $\iota$ of $x_i(t)$, let $\xi_q^i(t)\in T_{x_i(t)}G$ be uniquely determined by $q=\exp_{x_i(t)}\xi_q^i(t)$. For each multiindex $\ell=(i,j,k)\in\mathbbm N^3$ we define the smooth map
\begin{eqnarray*}
\V_{\ell}^+\colon\Omega G\to\mathbbm R,\quad x\mapsto\rho_k\big(\|x-x_i\|_{L^2(S^1)}^2\big)\int_0^1V_{ij}(t,x(t))\,dt,
\end{eqnarray*}
where
\begin{eqnarray*}
V_{ij}(t,q):=\begin{cases}
\beta(|\xi_q^i(t)|^2)\langle\eta_{ij}(t)+\xi_q^i(t),\eta_{ij}(t)\rangle,&\textrm{if}\,|\xi_q^i(t)|<\iota,\\
0,&\textrm{else}.
\end{cases}
\end{eqnarray*}

The $L^2$ distance appearing in the argument of $\rho_k$ above refers to the $L^2$ distance induced after isometrically embedding the manifold $G$ in some ambient euclidian space $\mathbbm R^N$. Note that $\V_{\ell}^+$ extends uniquely to a map $\V_{\ell}^+\colon\Lambda G\to\R$ on the free loop group, which is invariant under the free action $h\cdot x\mapsto hx$ of $G$ on $\Lambda G$.

\begin{rem}\label{rem:perturbationsEnonneg}
\upshape
Because $\langle\eta_{ij}+\xi_q^i,\eta_{ij}\rangle\geq0$ if $\|\xi_q^i\|_{L^2(S^1)}\leq\|\eta_{ij}\|_{L^2(S^1)}$, it follows that the map $\V_{\ell}$ is non-negative for sufficiently large indices $k$ (for given pair $(i,j)$). We henceforth consider only those multiindices $\ell=(i,j,k)$ which satisfy this condition, and renumber the set of such triples $(i,j,k)$ by integers $\ell\in\mathbbm N$.
\end{rem}

Let $Y^+$ denote the real vector space spanned by the maps $\V_{\ell}^+$, $\ell\in\mathbbm N$. We endow it with a norm analogously to that of the space $Y^-$ (cf.~the previous section), turning it into a separable Banach space isomorphic to the space $\ell^1$ of absolutely summable real sequences (cf.~\cite[Section 7.1]{Web} for details).

\subsection{Critical manifolds}\label{sect:critmanifolds}

We introduce some further notation concerning the set of critical points of the functionals $\YM$ and $\E$. Let
\begin{eqnarray*} 
\crit(\YM)\coloneqq\{A\in\A^{1,p}(P)\mid d_A^{\ast}F_A=0\}
\end{eqnarray*}
denote the set of critical points of $\YM$, the equation $d_A^{\ast}F_A=0$ being understood in the weak sense. Similarly, the notation $\crit(\YM^{\V^-})$ refers to the set of critical points of the perturbed Yang--Mills functional. We furthermore let
\begin{eqnarray*}
\crit(\E)\coloneqq\{x\in W^{1,p}(S^1,G)\mid\nabla_t\partial_tx=0\}
\end{eqnarray*}
denote the set of critical points of $\E$, where we again interpret the equation $\nabla_t\partial_tx=0$ in the weak sense. The set of critical points of the perturbed energy functional is denoted by $\crit(\E^{\V^+})$. We subsequently make use of the holonomy map $\Phi\colon\A(P)\to\Lambda G$ as given by \eqref{mapPhi}. A discussion of its properties is postponed to \ref{holmapappendix}. Let $\widehat{\CR}(\YM)$ and $\widehat{\CR}(\E)$ denote the set of connected components of $\crit(\YM)$, respectively of $\crit(\E)$. The group $\G_0(P)$ of based gauge transformations acts freely on $\A(P)$, hence on $\crit(\YM)$. We define $\CR(\YM)$ as the set of connected components of $\crit(\YM)/\G_0(P)$ in $\A(P)/\G_0(P)$. It is a standard fact that every such connected component is a finite-dimensional submanifold of $\A(P)/\G_0(P)$ diffeomorphic to some homogeneous space $G/H$. Likewise, the group $G$ acts freely on $\Lambda G$ and on $\widehat{\CR}(\E)$. We denote by $\CR(\E)=\widehat{\CR}(\E)/G$ the set of connected components of $\crit(\E)/G$ in $\Lambda G/G\cong\Omega G$. Furthermore, the holonomy map $\Phi$ induces a bijection between $\CR(\YM)$ and $\CR(\E)$ which preserves the action filtration given on both these sets, cf.~Theorem \ref{critcorr}. For the remainder of this section we fix a regular value $a$ of $\YM$ and put $b\coloneqq4a/\pi$, which is a regular value of $\E$. Furthermore, we denote
\begin{align*}
\crit^a(\YM):=&\{A\in\crit(\YM)\mid\YM(A)\leq a\},\\
\crit^b(\E):=&\{x\in\crit(\E)\mid\E(x)\leq b\}.
\end{align*}

We introduce the notation $\widehat{\CR}^a(\YM)$, $\CR^a(\YM)$, $\widehat{\CR}^b(\E)$, and $\CR^b(\E)$ for the intersection of $\widehat{\CR}(\YM)$ etc.~as defined above with the sublevel set $\crit^a(\YM)$, respectively $\crit^b(\E)$. As a consequence of Theorem \ref{thmdecompinequality} the holonomy map $\Phi$ induces a bijection between $\widehat{\CR}^a(\YM)$ and $\CR^b(\E)$. For each critical manifold $\hat{\mathcal C}\in\widehat{\CR}^a(\YM)$ we fix a closed $L^2$ neighborhood $U_{\hat{\mathcal C}}$ of $\hat{\mathcal C}$ such that $U_{\hat{\mathcal C}_1}\cap U_{\hat{\mathcal C}_2}=\emptyset$ whenever $\hat{\mathcal C}_1\neq\hat{\mathcal C}_2$. Because the set $\widehat{\CR}^a(\YM)$ is finite (cf.~\cite{Swoboda1} for a proof) it follows that such a choice is possible. We next choose for each $\mathcal C\in\CR^b(\E)$ a sufficiently small closed $L^2$ neighborhood $U_{\mathcal C}$ of $\mathcal C$ such that $U_{\mathcal C_1}\cap U_{\mathcal C_2}=\emptyset$ if $\mathcal C_1\neq\mathcal C_2$, and $\Phi(U_{\hat{\mathcal C}})\cap U_{\mathcal C}=\emptyset$ for all $\hat{\mathcal C}\in\widehat{\CR}^a(\YM)$ with $\Phi(\hat{\mathcal C})\neq\mathcal C$.

\begin{defn}\label{def:regperturbation}\upshape
Let $a$ be the regular value of $\YM$ as fixed above. We call a perturbation $\V^-=\sum_{\ell=1}^{\infty}\lambda_{\ell}\V_{\ell}^-\in Y^-$ \emph{$a$-admissible} if it satisfies 
\begin{eqnarray*}
\supp\V_{\ell}^-\cap U_{\hat{\mathcal C}}\neq\emptyset\quad\textrm{for some}\quad\hat{\mathcal C}\in\widehat{\CR}^a(\YM)\quad\Longrightarrow\quad\lambda_{\ell}=0.
\end{eqnarray*}
For the regular value $b=4a/\pi$ of $\E$ we analogously define the subspace of  \emph{$b$-admissible} perturbations of $Y^+$. We let $Y_a^-\times Y_b^+\eqqcolon Y_a\subseteq Y$ denote the subspace of pairs $(\V^-,\V^+)$ where $\V^-$ is $a$-admissible and $\V^+$ is $b$-admissible.
\end{defn}

It is straightforward to show that the spaces $Y_a^-$, respectively $Y_b^+$ are closed subspaces of the Banach spaces $Y^{\pm}$, and hence $Y_a$ is a closed subspace of $Y$. The following proposition shows that adding a small $a$-admissible perturbation to $\YM$ leaves the set of critical points below level $a$ unchanged.

\begin{prop}\label{prop:samecritpoints} 
Let $a$ be the regular value of $\YM$ as fixed above. Then there is a constant $\delta=\delta(a)>0$ with the following significance. Assume $\V^-\in Y_a^-$ and $\|\V^-\|<\delta$. Then it holds,
\begin{multline*}
\crit(\YM^{\V^-})\cap\{A\in\A^{1,p}(P)\mid\YM(A)<a\}\\
=\crit(\YM)\cap\{A\in\A^{1,p}(P)\mid\YM(A)<a\}.
\end{multline*}
\end{prop}

\begin{proof}
For a proof we refer to \cite[Proposition 2.9]{Swoboda1}. 
\end{proof}

The analogous statement holds true for the energy functional $\E$.

\begin{prop}\label{prop:samecritpointsE}
Let $b$ be the regular value of $\E$ as fixed above. Then there is a constant $\delta=\delta(b)>0$ with the following significance. Assume $\V^+\in Y_b^+$ and $\|\V^+\|<\delta$. Then,
\begin{eqnarray*}
\crit(\E^{\V^+})\cap\{x\in\Lambda G\mid\E(x)<b\}=\crit(\E)\cap\{x\in\Lambda G\mid \E(x)<b\}.
\end{eqnarray*}
\end{prop}

\begin{proof}
For a proof we refer to \cite[Lemma 7.4]{Web}.  
\end{proof}

\subsection{Gradient flow lines}
 
\setcounter{footnote}{0}

In this section we fix perturbations $\V^{\pm}\in Y^{\pm}$. As discussed in \cite{Swoboda1} it is convenient to introduce into the $\mathcal G(P)$-invariant Yang--Mills gradient flow equation an additional \emph{gauge fixing term} $-d_A\Psi$, $\Psi\in\Omega^0(\Sigma,\ad(P))$. Solutions of this new flow equation are then invariant under \emph{time-dependent} gauge transformations. This point requires to introduce some further notation. For an interval $I\subseteq\R$ we denote by $\hat P_I\coloneqq I\times P$ the trivial extension of the principle $G$-bundle $P$ to the base manifold $I\times\Sigma$, and set $\hat P\coloneqq\hat P_{\R^-}$. We let $\G(\hat P)$ denote the group of smooth gauge transformations of the principle $G$-bundle $\hat P$ and call this the group of time-dependent gauge-transformations (and similarly for $\G(\hat P_I)$). A pair $(A,\Psi)\in C^{\infty}(\R,\A(P)\times\Omega^0(\Sigma,\ad(P)))$ can naturally be identified with the connection $\AA=A+\Psi\,ds\in\A(\hat P)$. The action of the group $\G(\hat P)$ on $\A(\hat P)$ by gauge transformations is given by
\begin{eqnarray}\label{eq:extendedgaugetransf}
g^{\ast}\AA=g^{\ast}A+(g^{-1}\Psi g+g^{-1}\partial_sg)\wedge ds.
\end{eqnarray}

\begin{defn}\label{defnEYF}\upshape
Let $\V\in Y$ be a perturbation. The $\G(\hat P)$-\emph{invariant, perturbed Yang--Mills gradient flow} is the nonlinear PDE
\begin{eqnarray}\label{EYF}
0=\partial_sA+d_A^{\ast}F_A-d_A\Psi+\nabla\mathcal V^-(A)
\end{eqnarray}
for connections $\AA=A+\Psi\,ds\in\A(\hat P)$.
\end{defn}

One easily checks that with $\AA\in\A(\hat P)$, also $g^{\ast}\AA$ is a solution of \eqref{EYF}, for every $g\in\G(\hat P)$.

\begin{defn}\label{defnEEF}\upshape
The \emph{perturbed loop group gradient flow equation} is the nonlinear PDE
\begin{eqnarray}\label{EEF}
0=\partial_sx-\nabla_t\partial_tx+\nabla\V^+(x)
\end{eqnarray}
for a smooth path $x\colon s\mapsto x(s)\in\Lambda G$ of free loops. 
\end{defn}

Equation \eqref{EEF} is clearly invariant under the action of the group $G$ on $\Lambda G$ via $(h\cdot x)(t)=hx(t)$. The proper analytical setup for a study of the perturbed Yang--Mills and loop group gradient flow equations will be introduced in Section \ref{sect:nonlinsetup}.

\subsection{Morse homologies for the Yang--Mills and heat flows}\label{sect:Morsehomologies}

Morse homology theories based on the perturbed Yang--Mills and loop group gradient flows have been developed by Weber, respectively the author in \cite{Web, Swoboda1}. We here only give a brief description of both these Morse homology theories. Let $a\geq0$ be a regular value of $\YM$. We fix an $a$-admissible perturbation $\V^-\in Y_a^-$ (cf.~Definition \ref{def:regperturbation}). Let $h\colon\crit^a(\YM)/\mathcal G_0(P)\to\R$ be a smooth Morse--Smale function (with respect to some fixed Riemannian metric on the finite-dimensional smooth manifold $\crit^a(\YM)/\mathcal G_0(P)$). We let
\begin{eqnarray*}
CM_{\ast}^a(\A(P),\V^-,h)
\end{eqnarray*}
denote the complex generated as a $\mathbbm Z_2$ vector space by the set $\crit(h)$ of critical points of $h$. To each $x\in\crit(h)$ we assign the index $\Ind(x)$ to be the sum of the Morse indices of $x$ as a critical point of $\YM$ and of the function $h$. For $x^-,x^+\in\crit(h)$ we call the set $\mathcal M(x^-,x^+)$ as in \cite[Section 8.2]{Swoboda1} the moduli space of Yang--Mills gradient flow lines with cascades from $x^-$ to $x^+$.

\begin{lem}\label{lem:cascadmodspace}
For generic, $a$-admissible perturbation $\V^-\in Y_a^-$, Morse function $h$, and all $x^-,x^+\in\crit(h)$, the set $\mathcal M(x^-,x^+)$ is a smooth manifold (with boundary) of dimension
\begin{eqnarray*}
\dim\mathcal M(x^-,x^+)=\Ind(x^-)-\Ind(x^+)-1.
\end{eqnarray*}
\end{lem}

\begin{proof}
For a proof we refer to \cite[Lemma 8.3]{Swoboda1}.
\end{proof}

For $k\in\mathbbm N$ we define the Morse boundary operator
\begin{eqnarray*}
\partial_k^{\YM}\colon CM_k^a(\A(P),\V^-,h)\to CM_{k-1}^a(\A(P),\V^-,h)
\end{eqnarray*}
to be the linear extension of the map
\begin{eqnarray}\label{eq:boundaryoperator}
\partial_k^{\YM}x\coloneqq\sum_{x'\in\crit(h)\atop\Ind(x')=k-1}n(x,x')x',
\end{eqnarray}
where $x\in\crit(h)$ is a critical point of index $\Ind(x)=k$. The numbers $n(x,x')$ are given by counting modulo $2$ the flow lines with cascades (with respect to $\YM^{\V^-}$ and $h$) from $x$ to $x'$, i.e.
\begin{eqnarray*}
n(x,x')\coloneqq\#\,\mathcal M(x^-,x^+)\pmod 2.
\end{eqnarray*}

\begin{thm}[Yang--Mills Morse homology]\label{thm:mainresultSwoboda}
Let $a\geq0$ be a regular value of $\YM$. For any Morse function $h\colon\crit^a(\YM)/\mathcal G_0(P)\to\R$ and generic, $a$-admissible perturbation $\V^-\in Y_a^-$, the map $\partial_{\ast}^{\YM}$ satisfies $\partial_k^{\YM}\circ\partial_{k+1}^{\YM}=0$ for all $k\in\mathbbm N_0$ and thus there exist well-defined homology groups
\begin{eqnarray*}
HM_k^a(\A(P),\V^-,h)=\frac{\ker\partial_k^{\YM}}{\im\partial_{k+1}^{\YM}}.
\end{eqnarray*}
This Yang--Mills Morse homology is independent of the choice of $a$-admissible perturbation $\V^-$ and Morse function $h$.
\end{thm}

\begin{proof}
For a proof we refer to \cite[Theorem 1.1]{Swoboda1}.
\end{proof}

Weber's heat flow homology for the loop space $\Omega M$ of a closed manifold $M$ is based on a similar construction of a chain complex and a boundary operator. One of his main results is the following theorem (which he only states for the case where the action functional $\EV$ is Morse, the adaption to the present case of a Morse--Bott situation being straight-forward).

\begin{thm}[Heat flow homology]\label{thm:mainresultWeber}
Let $b\geq0$ be a regular value of $\EV$. For any Morse function $h\colon\crit^b(\E)/G\to\R$ and generic, $b$-admissible perturbation $\V^+\in Y_b^+$, the map $\partial_{\ast}^{\E}$ (defined in analogy to \eqref{eq:boundaryoperator}) satisfies $\partial_k^{\E}\circ\partial_{k+1}^{\E}=0$ for all $k\in\mathbbm N_0$ and thus there exist well-defined homology groups
\begin{eqnarray*}
HM_k^b(\Omega M,\V^+,h)=\frac{\ker\partial_k^{\E}}{\im\partial_{k+1}^{\E}}.
\end{eqnarray*}
The homology $HM_{\ast}^b(\Omega M,\V^+,h)$ is called {\bf{\emph{heat flow homology}}}. It is independent of the choice of $b$-admissible perturbation $\V^+$ and Morse function $h$.
\end{thm}

\begin{proof}
For a proof we refer to \cite[Theorem 1.14]{Web}.
\end{proof}

\subsection{Hybrid moduli spaces}\label{sect:Hybridmodspace}

The main idea in comparing Yang--Mills Morse homology to heat flow homology consists in the following construction of a so-called \emph{hybrid moduli space}. For abbreviation we set $A_0\coloneqq A(0)$, $x_0\coloneqq x(0)$, and $x_{A_0}\coloneqq\Phi(A(0))$. For critical manifolds $\hat{\mathcal C}^-\in\widehat{\CR}(\YM)$ and $\hat{\mathcal C}^+\in\widehat{\CR}(\E)$ let us define
\begin{align}\label{def:premodspace}
\lefteqn{\hat{\mathcal M}(\hat{\mathcal C}^-,\hat{\mathcal C}^+)\coloneqq}\\
\nonumber&\big\{(A,\Psi,x)\in C^{\infty}(\mathbbm R^-,\A(P)\times\Omega^0(\Sigma,\ad(P)))\times\,C^{\infty}(\mathbbm R^+,\Lambda G)\;\big|\\
\nonumber&(A,\Psi)\;\textrm{satisfies}\;\eqref{EYF},\quad x\;\textrm{satisfies}\;\eqref{EEF},\quad x_0=hx_{A_0}\;\textrm{for some}\;h\in G,\\
\nonumber&\lim_{s\to-\infty}(A(s),\Psi(s))=\big(A^-,0)\in\hat{\mathcal C}^-\times\Omega^0(\Sigma,\ad(P)),\lim_{s\to+\infty}x(s)=x^+\in\hat{\mathcal C}^+\big\}.
\end{align}

By construction, the group $\G(\hat P)\times G$ acts freely on the space $\hat{\mathcal M}(\hat{\mathcal C}^-,\hat{\mathcal C}^+)$. Let us denote $\mathcal C^-\coloneqq\hat{\mathcal C}^-/\G_0(P)$ and $\mathcal C^-\coloneqq\hat{\mathcal C}^+/G$. The moduli space we shall study further on is the quotient
\begin{eqnarray}\label{modulispace}
\mathcal M(\mathcal C^-,\mathcal C^+)\coloneqq\frac{\hat{\mathcal M}(\hat{\mathcal C}^-,\hat{\mathcal C}^+)}{\G(\hat P)\times G}.
\end{eqnarray}
 
We show subsequently that $\hat{\mathcal M}(\hat{\mathcal C}^-,\hat{\mathcal C}^+)$ arises as the zero set $\F^{-1}(0)$ of an equivariant (with respect to $\G(\hat P)\times G$) section $\F$ of a suitably defined Banach space bundle $\mathcal E$ over a Banach manifold $\mathcal B$. After proving that the horizontal differential $d_x\F$ at any such zero $x\in\F^{-1}(0)$ is a surjective Fredholm operator, it will follow from the implicit function theorem that the moduli space $\mathcal M(\mathcal C^-,\mathcal C^+)$ is a finite-dimensional smooth manifold.

\section{Fredholm theory and transversality}

\subsection{The nonlinear setup}\label{sect:nonlinsetup}
\setcounter{footnote}{0}

The Banach manifolds we shall use are modeled on weighted Sobolev spaces in order to make the Fredholm theory work. To define these, we choose numbers $\delta>0$ and $p>3$, and a smooth function $\beta$ such that $\beta(s)=-1$ if $s<0$ and $\beta(s)=1$ if $s>1$. We define the $\delta$-weighted $(k,p)$-Sobolev norm (for $1\leq p\leq\infty$ and an integer $k\geq0$) of a measurable function (respectively, a measurable section of a vector bundle) $u$ over $\R^-\times\Sigma$ or $\R^+\times S^1$ to be the usual $(k,p)$-Sobolev norm of the function (or section) $e^{\delta\beta(s)s}u$.\\
\noindent\\
Recall the definition of parabolic Sobolev spaces at the end of Section \ref{sec:prelim}. We fix numbers $\delta>0$, $p>3$. Let $\A_{\delta}^{1,2;p}(P)$ denote the space of time-dependent connections on $P$ which are locally of class $W^{1,2;p}$ and for which there exists a limiting connection $A^-\in\crit(\YM)$ and a number $T^-\leq0$ such that the time-dependent $1$-form $\alpha^-\coloneqq A-A^-$ satisfies
\begin{eqnarray}\label{eq:regularityatend1}
\alpha^-\in W_{\delta}^{1,p}((-\infty,T^-],L^p(\Sigma,\ad(P)))\cap L_{\delta}^p((-\infty,T^-],W^{2,p}(\Sigma,\ad(P))).
\end{eqnarray}

Similarly, let $\G_{\delta}^{2,p}(\hat P)$ denote the group of gauge transformations of $\hat P$ which are locally of class $W^{2,p}$ and in addition satisfy the following two conditions. First, the time-dependent $\ad(P)$-valued $1$-form $g^{-1}dg$ satisfies 
\begin{eqnarray*}
g^{-1}dg\in L_{\delta}^p(\R,W^{2,p}(\Sigma,T^{\ast}\Sigma\otimes\ad(P)))
\end{eqnarray*}
(this condition being necessary to make sure that $g^{\ast}A\in L_{\delta}^p(\R,\A^{2,p}(P))$ for every $A$ with this property). Second, there exists a limiting gauge transformation $g^-\in\G_0^{2,p}(P)$, a number $T^-\leq0$, and a bundle valued $0$-form
\begin{eqnarray*}
\gamma^-\in W_{\delta}^{2,p}((-\infty,T^-]\times\Sigma,\ad(\hat P_{(-\infty,T^-]}))
\end{eqnarray*}
with $g(s)=g^-\exp(\gamma^-(s))$ for $s\leq T^-$.\\
Let $\hat{\mathcal C}^-\in\widehat{\CR}(\YM)$ be a critical manifold. We denote by $\hat\B^-\coloneqq\hat\B^-(\hat{\mathcal C}^-,\delta,p)$ the Banach manifold of pairs
\begin{eqnarray*}
(A,\Psi)\in\A_{\delta}^{1,2;p}(P)\times W_{\delta}^{1,p}(\R^-\times\Sigma)
\end{eqnarray*}
such that $\lim_{s\to-\infty}A(s)=A^-$ holds in the sense of \eqref{eq:regularityatend1} for some $A^-\in\hat{\mathcal C}^-$. We identify such pairs as before with connections $\AA=A+\Psi\,ds$ on $\hat P$. The action of the group $\G_{\delta}^{2,p}(\hat P)$ on $\hat\B$ by gauge transformations as in \eqref{eq:extendedgaugetransf} is smooth. It is free by our requirement that $\lim_{s\to-\infty}g(s)=g^-$ be a based gauge transformation. The resulting quotient space
\begin{eqnarray*}
\B^-\coloneqq\B^-(\mathcal C^-,\delta,p)\coloneqq\frac{\hat\B^-(\hat{\mathcal C}^-,\delta,p)}{\G_{\delta}^{2,p}(\hat P)}
\end{eqnarray*}
is again a smooth Banach manifold. Let $\hat{\mathcal C}^+\in\widehat{\CR}(\E)$. We define $\hat{\mathcal B}^+(\hat{\mathcal C}^+,\delta,p)$ to be the Banach manifolds of maps $x\colon\R^+\times S^1\to G$ such that the condition
\begin{eqnarray*}
x-x^+\in L_{\delta}^p(\mathbbm R^+,W^{2,p}(S^1,G))\cap\,W_{\delta}^{1,p}(\mathbbm R^+,L^p(S^1,G))
\end{eqnarray*}
is satisfied for some $x^+\in\hat{\mathcal C}^+$. To make sense of the difference $x-x^+$ and of the Sobolev spaces involved in this definition, we consider the Lie group $G$ as being isometrically embedded in some euclidian space $\mathbbm R^N$. Now put 
\begin{eqnarray*}
\mathcal B^+\coloneqq\mathcal B^+(\mathcal C^+,\delta,p)\coloneqq\frac{\hat{\mathcal B}^+(\hat{\mathcal C}^+,\delta,p)}{G},
\end{eqnarray*}
and define $\mathcal B\coloneqq\mathcal B^-\times\mathcal B^+$. We define the Banach space bundle $\mathcal E=\mathcal E(\mathcal C^-,\mathcal C^+,\delta,p)$ over $\mathcal B$ in the following way. Let $\hat\E^-$ be the trivial Banach space bundle over $\hat\B^-$ with fibres $\hat{\mathcal E}_{(A,\Psi)}^-\coloneqq L_{\delta}^p(\mathbbm R^-,L^p(\Sigma,T^{\ast}\Sigma\otimes\ad(P)))$, and $\hat{\mathcal E}^+$ be the trivial Banach space bundle over $\hat{\mathcal B}^+$ with fibres $\hat{\mathcal E}_x^+\coloneqq L_{\delta}^p(\mathbbm R^+,L^p(S^1,\mathfrak g))$. We set
\begin{eqnarray*}
\hat\E\coloneqq\hat\E(\hat{\mathcal C}^-,\hat{\mathcal C}^+,\delta,p)\coloneqq\hat\E^-\times\hat{\mathcal E}^+\times\Lambda G.
\end{eqnarray*}

The action of $\G_{\delta}^{2,p}(\hat P)\times G$ on $\hat\B^-\times\hat\B^+$ lifts to a free action on $\hat\E^-\times\hat{\mathcal E}^+\times\Lambda G$ via
\begin{eqnarray*}
(g,h)\cdot(A,\Psi,\alpha,x,\xi,x_1)\coloneqq(g^{\ast}A,g^{-1}\Psi g+g^{-1}\partial_sg,g^{-1}\alpha g,hx,\xi,hx_1).
\end{eqnarray*}

Let $\E$ denote the respective quotient space and define the smooth section $\F\colon\mathcal B\to\mathcal E$ of $\mathcal E$ by
\begin{eqnarray}\label{sectionF}
\F\colon[(A,\Psi,x)]\mapsto\left[\left(\begin{array}{c}\partial_sA+d_A^{\ast}F_A-d_A\Psi+\nabla\V^-(A)\\
 x^{-1}\left(\partial_sx-\nabla_t\partial_tx+\nabla\V^+(x)\right)\\
 x(0)x_{A(0)}^{-1}
\end{array}\right)\right].
\end{eqnarray}

\subsection{Yang--Mills Hessian and linearized Yang--Mills gradient flow}

For $A\in\A(P)$, we let $\He_A$ denote the {\emph{augmented Yang--Mills Hessian}} defined by 
\begin{multline}\label{def:augmentedYMH}
\He_A\coloneqq\left(\begin{array}{cc}d_A^{\ast}d_A+\ast[\ast F_A\wedge\,\cdot\,]+H_A\V^-&-d_A\\-d_A^{\ast}&0\end{array}\right)\colon\\
\Omega^1(\Sigma,\ad(P))\oplus\Omega^0(\Sigma,\ad(P))\to\Omega^1(\Sigma,\ad(P))\oplus\Omega^0(\Sigma,\ad(P)).
\end{multline}

Here $H_A\V^-$ denotes the Hessian of the map $\V^-\colon\A(P)\to\R$. 
In order to find a domain for $\H_A$ which makes the subsequent Fredholm theory work, we fix an irreducible smooth reference connection $A_0\in\A(P)$ and decompose the space $\Omega^1(\Sigma,\ad(P))$ of smooth $\ad(P)$-valued $1$-forms as the $L^2(\Sigma)$ orthogonal sum
\begin{align*} 
\Omega^1(\Sigma,\ad(P))=&\ker\big(d_{A_0}^{\ast}\colon\Omega^1(\Sigma,\ad(P))\to\Omega^0(\Sigma,\ad(P))\big)\\
&\oplus\;\im\big(d_{A_0}\colon\Omega^0(\Sigma,\ad(P))\to\Omega^1(\Sigma,\ad(P))\big).
\end{align*}

Let $W_0^{2,p}$ and $W_1^{1,p}$ denote the completions of the first component, respectively of the second component with respect to the $(k,p)$-Sobolev norm ($k=1,2$). We set $\mathcal W^p(\Sigma)\coloneqq W_0^{2,p}\oplus W_1^{1,p}$ and endow this space with the sum norm. It was shown in \cite[Proposition 5.1]{Swoboda1} that it is independent of the choice of reference connection $A_0$. We let $p>1$ and consider the augmented Yang--Mills Hessian as an operator
\begin{eqnarray*}
\H_A\colon\mathcal W^p(\Sigma)\oplus W^{1,p}(\Sigma,\ad(P))\to L^p(\Sigma,T^{\ast}\Sigma\otimes\ad(P))\oplus L^p(\Sigma,\ad(P)).
\end{eqnarray*}
In the case $p=2$ this is a densely defined symmetric operator on the Hilbert space $L^2(\Sigma,T^{\ast}\Sigma\otimes\ad(P))\oplus L^2(\Sigma,T^{\ast}\Sigma)$ with domain
\begin{eqnarray}\label{def:domainHA}
\dom\H_A\coloneqq\mathcal W^2(\Sigma)\oplus W^{1,2}(\Sigma,\ad(P)).
\end{eqnarray}
It was shown in \cite[Proposition 5.2]{Swoboda1} that it is self-adjoint.

Next we consider the linearization of the Yang--Mills gradient flow \eqref{EYF}. Since every solution $(A,\Psi)$ of the Yang--Mills gradient flow is gauge equivalent under $\G(\hat P)$ to a solution satisfying $\Psi\equiv0$ (cf.~\cite[Proposition 3.3]{Swoboda}), it suffices to consider the linearization along such trajectories only. We define for $p>1$ the Banach spaces 
\begin{align*}
\Z^{\delta,p,-}\coloneqq&\big(W_{\delta}^{1,p}(\R,L^p(\Sigma,T^{\ast}\Sigma\otimes\ad(P)))\cap L_{\delta}^p(\R,\mathcal W^p(\Sigma))\big)
\oplus W_{\delta}^{1,p}(\R\times\Sigma,\ad(P)),\\
\mathcal L^{\delta,p,-}\coloneqq&L_{\delta}^p(\mathbbm R^-\times\Sigma,T^{\ast}\Sigma\otimes\ad(P))\oplus L_{\delta}^p(\mathbbm R^-\times\Sigma,\ad(P)),
\end{align*}
where the number $\delta>0$ refers to the weight function fixed at the beginning of Section \ref{sect:nonlinsetup}. In the following we shall be concerned with the linear operator 
\begin{eqnarray}\label{def:operatorDA}
\mathcal D_A\coloneqq\frac{d}{ds}+\mathcal H_A\colon\Z^{\delta,p,-}\to\mathcal L^{\delta,p,-}
\end{eqnarray}
for a smooth path $s\mapsto A(s)\in\A(P)$, where $s\in\R^-$. It arises as the linearization of the Yang--Mills gradient flow \eqref{EYF} at a solution $(A,\Psi)=(A,0)$. Some of its properties are collected in \ref{sect:aprioriestimates}.

\subsection{Linearized loop group gradient flow} 
\setcounter{footnote}{0}

We discuss the linearized loop group gradient flow, following Weber \cite{Web}. The Hessian of the energy functional $\E$ at the loop $x\in\Lambda G$ is the linear operator
\begin{eqnarray*} 
H_x\colon\xi\mapsto\nabla_t\nabla_t\xi+R(\xi,\partial_tx)\partial_tx+H_x\V^+\xi
\end{eqnarray*} 
for vector fields $\xi$ along $x$. Here $R$ denotes the Riemannian curvature tensor of $G$, and $H_x\V^+$ denotes the Hessian of the map $\V^+\colon\Lambda G\to\R$. Let $x\colon\R^+\times S^1\to G$ be a smooth map. We define for $p>1$ the Banach spaces 
\begin{align*}
\Z^{\delta,p,+}\coloneqq&W_{\delta}^{1,p}(\mathbbm R^+,L^p(S^1,x^{\ast}TG))\cap W_{\delta}^p(\mathbbm R^+,W^{2,p}(S^1,x^{\ast}TG)),\\
\mathcal L^{\delta,p,+}\coloneqq&L_{\delta}^p(\R^+\times S^1,x^{\ast}TG).
\end{align*}

Note that these depend on $x$, which is suppressed in our notation. For short we will often drop $x^{\ast}TG$ and simply write $L^p(S^1)$ etc. The number $\delta>0$ refers to the weight function fixed at the beginning of Section \ref{sect:nonlinsetup}. 
We denote
\begin{eqnarray*}
\mathcal D_x\coloneqq\frac{d}{ds}+\mathcal H_x\colon\Z^{\delta,p,+}\to\mathcal L^{\delta,p,+}
\end{eqnarray*}
for a smooth path $s\mapsto x(s)\in\Lambda G$, where $s\in\R^+$. Note that the operator $\mathcal D_x$ arises as the linearization of the loop group gradient flow \eqref{EEF}. We discuss some of its properties in \ref{sect:aprioriestimates}.

\subsection{Linearized moduli space problem}\label{subsect:linearizedoperator}

We fix numbers $p>3$, $\delta>0$ and regular values $a>0$ and $b\coloneqq\frac{4a}{\pi}$ of $\YM$, respectively of $\E$. Throughout we let $\V=(\V^-,\V^+)\in Y_a$ be an $a$-admissible perturbation with $\|\V\|$ sufficiently small such that Propositions \ref{prop:samecritpoints} and \ref{prop:samecritpointsE} apply. Let $\hat{\mathcal C}^-\in\mathcal{CR}^a(\YM)$ and $\hat{\mathcal C}^+\in\mathcal{CR}^b(\E)$ be critical manifolds. Recall the definition of $\hat{\mathcal M}(\hat{\mathcal C}^-,\hat{\mathcal C}^+)$ in \eqref{def:premodspace}. For $u=(A,\Psi,x)\in\hat{\mathcal M}(\hat{\mathcal C}^-,\hat{\mathcal C}^+)$ we define the Banach spaces
\begin{eqnarray*}
\Z^{\delta,p}\coloneqq\Z^{\delta,p,-}\oplus\Z^{\delta,p,+},\qquad\mathcal L^{\delta,p}\coloneqq\mathcal L^{\delta,p,-}\oplus\mathcal L^{\delta,p,+}\oplus L^2(S^1,\mathfrak g).  
\end{eqnarray*}

Put $A_0\coloneqq A(0)$ (and analogously for $\alpha_0$, $x_0$, and $\xi_0$). We use the notation
\begin{eqnarray*}
D\Phi_{A_0}\coloneqq\Phi(A_0)^{-1}\di\Phi(A_0)\colon\Omega^1(\Sigma,\ad(P))\to C^{\infty}(S^1,\mathfrak g),
\end{eqnarray*}
where $\Phi$ is the holonomy map as in \eqref{mapPhi}. We define the linear operator $\mathcal D_{u}\colon\Z^{\delta,p}\rightarrow\mathcal L^{\delta,p}$ by
\begin{eqnarray}\label{definitionD_u}
(\alpha,\psi,\xi)\mapsto(\mathcal D_A(\alpha,\psi),\mathcal D_x\xi,N_A(\alpha,\xi)),
\end{eqnarray}
where we denote $N_A(\alpha,\xi)\coloneqq x_0^{-1}\xi_0-D\Phi_{A_0}\alpha_0$. The following remark clarifies the relation between the operator $\mathcal D_{u}$ and the linearization of the section $\F$ in \eqref{sectionF}.

\begin{rem}\label{remBanachtangent}
\upshape
\begin{compactenum}[(i)]
\item
From the definition of $\F$ as a section $\F\colon\B\to\E$ (cf.~\eqref{sectionF}) it follows that its linearization $d\F(u)$, where $u=(A,\Psi,x)$, acts on the space of pairs $(\alpha,\psi,\xi)$ where $\alpha(s)$ converges exponentially to some $\alpha^-\in T_{A^-}\hat{\mathcal C}^-$ as $s\to-\infty$, and likewise $\xi(s)\to\xi^+\in T_{x^+}\hat{\mathcal C}^+$ as $s\to\infty$. This asymptotic behaviour is in slight contrast to that required for elements of $\Z^{\delta,p}$, the domain of the operator $\mathcal D_u$. However, it is easy to see that $d\F(u)$ is Fredholm if and only if this property holds for $\mathcal D_u$, and that the Fredholm indices are related via the formula
\begin{eqnarray*}
\ind d\F(u)=\ind\mathcal D_u+\dim\mathcal C^-+\dim\mathcal C^+.
\end{eqnarray*}
To see this, we view $d\F(u)$ as a compact perturbation of the operator $\mathcal D_u$, the latter being extended trivially to $\Z^{\delta,p}\oplus\R^{\dim\mathcal C^-}\oplus\R^{\dim\mathcal C^+}$.
\item
The operator $\mathcal D_u$ arises as the linearization of the \emph{unperturbed} Yang--Mills gradient flow equation $\eqref{EYF}$. The Fredholm theory for general perturbations $\V\in Y$ can be reduced to the unperturbed case because the terms involving $\V$ contribute only compact perturbations to the operator $\mathcal D_u$. 
\end{compactenum}
\end{rem}

Because the zero eigenspaces of the Hessians $\He_{A^-}$ and $H_{x^+}$ are in general (i.e.~if $\dim\mathcal C^{\pm}\geq1$) non-trivial, we cannot directly refer to standard theorems on the spectral flow to prove Theorem \ref{Fredholmtheorem} below. As an intermediate step, we instead use the Banach space isomorphisms 
\begin{align*}
\nu_1^-\colon&\Z^{\delta,p,-}\to\Z^{0,p,-}\eqqcolon\Z^{p,-},\qquad\nu_2^-\colon\mathcal L^{\delta,p,-}\to\mathcal L^{0,p,-}\eqqcolon\mathcal L^{p,-},\\
\nu_1^+\colon&\Z^{\delta,p,+}\to\Z^{0,p,+}\eqqcolon\Z^{p,+},\qquad\nu_2^+\colon\mathcal L^{\delta,p,+}\to\mathcal L^{0,p,+}\eqqcolon\mathcal L^{p,+}
\end{align*}
given by multiplication with the weight function $e^{\delta\beta(s)s}$ (with $\beta\colon\R\to\R$ the function introduced at the beginning of Section \ref{sect:nonlinsetup}). In the sequel we assume that the weight $\delta$ satisfies $0<\delta<\delta_0(\mathcal C^-,\mathcal C^+)$ where the positive constant $\delta_0(\mathcal C^-,\mathcal C^+)$ is defined to be the infimum of the set 
\begin{multline*}
\big\{|\lambda|\in\R\,\big|\,\lambda\neq0\;\textrm{is eigenvalue of}\;\mathcal H_A\;\textrm{for some}\;A\in\hat{\mathcal C}^-\;\textrm{or}\\
\lambda\neq0\;\textrm{is eigenvalue of}\;H_x\;\textrm{for some}\;x\in\hat{\mathcal C}^+\;\big\}.
\end{multline*}

We furthermore denote
\begin{eqnarray*}
\mathcal D_A^{\delta}\coloneqq\nu_2^-\circ\mathcal D_A\circ(\nu_1^-)^{-1}\qquad\textrm{and}\qquad\mathcal D_x^{\delta}\coloneqq\nu_2^+\circ\mathcal D_x\circ(\nu_1^+)^{-1}.
\end{eqnarray*}

We use the notation $\Z^p\coloneqq\Z^{p,-}\times\Z^{p,+}$ and $\mathcal L^p\coloneqq\mathcal L^{p,-}\times\mathcal L^{p,+}\times L^2(S^1,\mathfrak g)$, and set $\mathcal D_u^{\delta}\coloneqq(\mathcal D_A^{\delta},\mathcal D_x^{\delta},N_A)\colon\Z^p\to\mathcal L^p$. It is easy to check that the operator $\mathcal D_u$ is Fredholm if and only if this holds for $\mathcal D_u^{\delta}$, in which case both Fredholm indices coincide. Note also that the operator $\mathcal D_A^{\delta}$ takes the form
\begin{eqnarray}\label{eq:DAdelta}
\mathcal D_A^{\delta}=\frac{d}{ds}+\mathcal H_A-(\beta+\beta's)\delta,
\end{eqnarray}
and hence, if $\delta>0$ is chosen sufficiently small, the operator family $s\mapsto\mathcal H_A-(\beta+\beta's)\delta$ converges to the invertible operator $\mathcal H_{A^-}+\delta$ as $s\to-\infty$. Analogously, we have that  
\begin{eqnarray}\label{eq:Dgdelta}
\mathcal D_x^{\delta}=\frac{d}{ds}+H_x-(\beta+\beta's)\delta.
\end{eqnarray}

Here the operator family $s\mapsto H_x-(\beta+\beta's)\delta$ converges to the invertible operator $H_{x^+}-\delta$ as $s\to\infty$.

\subsection{Fredholm theorem}\label{sect:Fredholmthm}

For short, we use notation like $L^p(I)\coloneqq L^p(I,L^p(\Sigma,T^{\ast}\Sigma\otimes\ad(P))$ to denote the $L^p$ space of $\ad(P)$-valued $1$-forms over $I\times\Sigma$, where $I$ is some interval.

\begin{lem}\label{Fredholmtheorem1}
Let $u=(A,\Phi,x)\in\hat{\mathcal M}(\hat{\mathcal C}^-,\hat{\mathcal C}^+)$. There exist positive constants $c(u)$ and $T(u)$ such that the estimate 
\begin{multline}\label{eq:basicFredholminequ}
\|(\alpha,\psi,\xi)\|_{\Z^p}\leq c(u)\big(\|\mathcal D_u^{\delta}(\alpha,\psi,\xi)\|_{\mathcal L^p}+\|R(\alpha,\psi)\|_{\mathcal L^{p,-}}\\
+\|(\alpha,\psi)\|_{L^p([-T(u),0])}+\|\xi\|_{L^p([0,T(u)])}\big)
\end{multline}
is satisfied for all $(\alpha,\psi,\xi)\in\Z^p$. Here $R$ denotes the compact operator as in Lemma \ref{DAsemiFredholm}. As a consequence, the operator $\mathcal D_u^{\delta}$ has finite-dimensional kernel and closed range.
\end{lem}

\begin{proof}
Inequality \eqref{eq:basicFredholminequ} is invariant under gauge transformations in $\G(\hat P)$, and thus it suffices to prove it for $\Phi=0$. Set $\zeta\coloneqq(\alpha,\psi)$. For $T=T(A)>0$ large enough we choose a smooth cut-off function $\chi\colon(-\infty,0]\to\mathbbm R$ with support in $[-T,0]$ and such that $\chi(s)=1$ for $s\in[-T+1,0]$. Then with $R$ as in Lemma \ref{DAsemiFredholm} it follows for some constant $c(A,p)$ that 
\begin{multline}\label{estimatecutoff}
\|\zeta\|_{\Z^{p,-}}\leq\|\chi\zeta\|_{\Z^{p,-}}+\|(1-\chi)\zeta\|_{\Z^{p,-}}\\
\leq c(A,p)\big(\|\mathcal D_A^{\delta}\zeta\|_{L^p([-T,0])}+\|\zeta\|_{L^p([-T,0])}+\|\mathcal D_A^{\delta}\zeta\|_{\mathcal L^{p,-}}+\|R\zeta\|_{\mathcal L^{p,-}}\big).
\end{multline}
Here we used \eqref{apriori2A} to bound the term $\|\chi\zeta\|_{\Z^{p,-}}$. The estimate for the term $\|(1-\chi)\zeta\|_{\Z^{p,-}}$ follows from Lemma \ref{DAsemiFredholm}, which remains valid for any $\mathcal D_A^{\delta}$ sufficiently close (in operator norm) to $\mathcal D_{A^-}^{\delta}$. This property holds by Lemma \ref{continuityproperiesDA}. We now apply Lemma \ref{lem:linestimateDg} to bound the term $\xi$. Let $T=T(x)>0$ be as in the lemma. This yields for some constant $c=c(x,p,\delta)>0$ that
\begin{eqnarray}\label{eq:thmfinkernelZplus}
\|\xi\|_{\Z^{p,+}}\leq c\big(\|\mathcal D_x^{\delta}\xi\|_{\mathcal L^{p,+}}+\|\xi\|_{L^p([0,T],L^2(S^1))}+\|\xi(0)\|_{L^2(S^1)}\big).
\end{eqnarray}

From the definition of $N_A$ and boundedness of the operator $\di\Phi(A_0)\colon W^{1,2}(\Sigma)\to W^{1,2}(S^1)$ cf.~\cite[Lemma A.1]{Swoboda}, we can further estimate
\begin{eqnarray}\label{eq:thmfinkernelZplus1}
\|\xi(0)\|_{L^2(S^1)}\leq c\big(\|N_A(\alpha,\xi)\|_{L^2(S^1)}+\|\alpha(0)\|_{W^{1,2}(\Sigma)}\big).
\end{eqnarray}

To control the term $\|\alpha(0)\|_{W^{1,2}(\Sigma)}$ we use Lemma \ref{interpolationlemma} with $H=W^{1,2}(\Sigma)$, $V=W^{2,2}(\Sigma)$, and $V^{\ast}=L^2(\Sigma)$. Choosing $\eps>0$ in that lemma sufficiently large this implies for a constant $c(A,p)>0$ the estimate
\begin{eqnarray}\label{eq:thmfinkernelinterp}
\nonumber\|\alpha(0)\|_{W^{1,2}(\Sigma)}^p&\leq&\|\zeta(0)\|_{W^{1,2}(\Sigma)}^p\\
\nonumber&\leq&c(A,p)\Big(\int_{-\infty}^0\|\mathcal D_A^{\delta}\zeta\|_{L^2(\Sigma)}^p\,ds+\int_{-\infty}^0\|\zeta\|_{W^{1,2}(\Sigma)}^p\,ds\Big)\\
&\leq&c(A,p)\big(\|\mathcal D_A^{\delta}\zeta\|_{\mathcal L^{p,-}}^p+\|\zeta\|_{\mathcal Z^{p,-}}^p\big).
\end{eqnarray}
The last line follows, as by definition the norm of $\mathcal L^{p,-}$ is stronger than that of $L^p(\mathbbm R^-,L^2 (\Sigma))$ , and the norm of $\Z^{p,-}$ is stronger than that of $L^p(\mathbbm R^-,W^{1,2}(\Sigma))$. Combining estimates \eqref{eq:thmfinkernelZplus}, \eqref{eq:thmfinkernelZplus1}, \eqref{eq:thmfinkernelinterp} with \eqref{estimatecutoff} yields the claimed inequality \eqref{eq:basicFredholminequ} (with $T(u)\coloneqq\max\{T(A),T(x)\}$). To finish the proof, we note that the operator $\mathcal D_u^{\delta}=(\mathcal D_A^{\delta},\mathcal D_x^{\delta},N_A)$ is bounded and that the operator $R$ and the inclusion maps $\Z^{p,-}([-T(u),0])\hookrightarrow L^p([-T(u),0])$ and $\mathcal Z^{p,+}([0,T(u)])\hookrightarrow L^p([0,T(u)])$ are compact (by Rellich's theorem). Hence the assertions on the kernel and the range follow from the abstract closed range lemma (cf.~\cite[p.~14]{Sal2}). The proof is complete.
\end{proof} 
 
We now state and prove the main result concerning the linear operator $\mathcal D_u$.

\begin{thm}[Fredholm theorem]\label{Fredholmtheorem}
The operator $\mathcal D_u$ is a Fredholm operator of index 
\begin{eqnarray*} 
\ind\mathcal D_u=\ind A^--\ind x^+-\dim\mathcal C^+.
\end{eqnarray*}  
(Here $\ind A^-$, respectively $\ind x^+$, denotes the number of negative eigenvalues of $\mathcal H_{A^-}$ and $H_{x^+}$, counted with multiplicities).
\end{thm}

\begin{proof}
As pointed out after Remark \ref{remBanachtangent} it sufficies to prove the assertion for the operator $\mathcal D_u^{\delta}$, for any sufficiently small weight $\delta>0$. That the operator $\mathcal D_u^{\delta}$ has finite-dimensional kernel and closed range is part of Lemma \ref{Fredholmtheorem1}. It remains to establish the formula for the index.
Let $H\coloneqq L^2(S^1,x^{\ast}TG)$ and denote  
\begin{eqnarray*} 
S\coloneqq\{\xi_0\in H\mid\exists\xi\in\Z^{p,+}\;\textrm{s.t.}\;\mathcal D_x^{\delta}\xi=0\;\textrm{and}\;\xi(0)=\xi_0\}.
\end{eqnarray*}
Note that $S$ is a closed subspace of $H$. Let $T$ be the orthogonal complement of $S$ in $H$. We set $K\coloneqq\{(\alpha(0),\psi(0))\mid(\alpha,\psi)\in\ker\mathcal D_{(A,\Psi)}^{\delta}\}$. It follows that the kernel of the operator $\mathcal D_u^{\delta}$ has dimension
\begin{eqnarray}\label{eq:dimker}
\dim\ker\mathcal D_u^{\delta}=\dim\ker\di\Phi|_K+\dim(N(K,0)\cap S).
\end{eqnarray}
On the other hand, because the operators $\mathcal D_A^{\delta}$ and $\mathcal D_x^{\delta}$ are surjective, the dimension of $\coker\mathcal D_u^{\delta}$ equals the codimension of the affine space
\begin{eqnarray*}
W=\big\{N(\alpha(0),\xi(0))\,\big|\,\exists(\alpha,\psi,\xi)\in\Z^p\;\textrm{s.t.}\;
(\mathcal D_A^{\delta}(\alpha,\psi),\mathcal D_x^{\delta}\xi)=(\beta,\omega,\eta)\big\}
\end{eqnarray*}
for arbitrary but fixed $(\beta,\omega,\eta)\in\mathcal L^p$. Let $(\alpha,\psi)$ vary over the space $\ker\mathcal D_A^{\delta}$ to see that this codimension is given by
\begin{eqnarray}\label{eq:dimcoker}
\codim W=\dim T-\dim K+\dim\ker \di\Phi|_K+\dim(N(K,0)\cap S).
\end{eqnarray}
Combining \eqref{eq:dimker} and \eqref{eq:dimcoker} and using that $\dim K=\ind\mathcal H_{A^-}$ (by Lemma \ref{DAsurjective}) and $\dim T=\ind x^++\dim\mathcal C^+$ (by Lemma \ref{Dgsurjective}), the asserted index formula follows.
\end{proof}

\begin{rem}\label{eq:finalindex}
\upshape
In view of Remark \ref{remBanachtangent} and Theorem \ref{Fredholmtheorem} we obtain the formula
\begin{eqnarray}\label{eq:FredholmindexdF}
\ind d\F(u)=\ind A^--\ind x^++\dim\mathcal C^-
\end{eqnarray}  
for the Fredholm index of the linearization of $\F$.
\end{rem}

\subsection{Transversality}

Throughout we fix a regular value $a\geq0$ of $\YM$. Our aim here is to show that for every pair $(\mathcal C^-,\mathcal C^+)\in\CR^a(\YM)\times\CR^b(\E)$ (where $b=4a/\pi$) and a residual subset of $a$-admmissible perturbations $\V=(\V^-,\V^+)\in Y_a$ (cf.~Definition \ref{def:regperturbation}) the linearized section $d\F(u)$ is surjective, for all $u\in\mathcal M(\mathcal C^-,\mathcal C^+)$. Recall the definition of the Banach manifold $\mathcal B$, the Banach space bundle $\mathcal E$, and the section $\F$ (cf.~Section \ref{sect:nonlinsetup}). We now change our notation slightly and let $\mathcal M(\mathcal C^-,\mathcal C^+;\V)$ indicate the moduli space as in \eqref{modulispace}, i.e.~defined for a fixed perturbation $\V\in Y_a$. Throughout the rest of this section we also replace the notation $\F$ by $\F_{\V}$. Let then $\hat{\F}\colon\mathcal B\times Y_a\to\mathcal E$ denote the section of the Banach space bundle $\mathcal E$ defined by 
\begin{eqnarray}\label{eq:defnmaphatf}
\hat{\F}\colon[(A,\Psi,x,\V)]\mapsto\F_{\V}([(A,\Psi,x)]).
\end{eqnarray}

We call the zero set $\mathcal M^{\univ}(\mathcal C^-,\mathcal C^+)\coloneqq\{w\in\mathcal B\times Y_a\mid\hat{\F}(w)=0\}$ the \emph{universal moduli space}. Thus the perturbation $\V\in Y_a$ which had previously been kept fixed is now allowed to vary over the Banach space $Y_a$. In this section, our main result is the following.

\begin{thm}\label{thm:transversality}
There exists a residual subset $Y_a^{\reg}\subseteq Y_a$ of perturbations such that for every $\V\in Y_a^{\reg}$ and every $(\mathcal C^-,\mathcal C^+)\in\CR^a(\YM)\times\CR^b(\E)$ the moduli space $\mathcal M(\mathcal C^-,\mathcal C^+;\V)$ is a Banach submanifold of $\mathcal M^{\univ}(\mathcal C^-,\mathcal C^+)$.  
\end{thm}

\begin{proof}
Let $(\mathcal C^-,\mathcal C^+)\in\CR^a(\YM)\times\CR^b(\E)$. As shown in Theorem \ref{thm:linsurjetivederef} below, the linearized operator $d\hat{\F}(w)$ is surjective, for every $w\in\mathcal M^{\univ}(\mathcal C^-,\mathcal C^+)$. It hence follows from the implicit function theorem that $\mathcal M^{\univ}(\mathcal C^-,\mathcal C^+)$ is a smooth Banach manifold. The claim now follows from the Sard--Smale theorem for Fredholm maps between Banach manifolds. Namely, by standard results (cf.~\cite[Proposition 3.3]{Weber0}) it follows that the projection map $\pi\colon\mathcal M^{\univ}(\mathcal C^-,\mathcal C^+)\to Y_a$ is a smooth Fredholm map between Banach manifolds (with index equal to that of $\mathcal D_A$). Hence by \cite[Theorem 3.6.15]{Abraham} the set of regular values 
\begin{eqnarray*}
Y_a^{\reg}(\mathcal C^-,\mathcal C^+)\coloneqq\big\{\V\in Y_a\,\big|\,d\pi(w)\,\textrm{is surjective for all}\,w\in\mathcal M(\mathcal C^-,\mathcal C^+;\V)\big\} 
\end{eqnarray*}
of $\pi$ is residual in $Y_a$. Again by the implicit function theorem, it follows that $\mathcal M(\mathcal C^-,\mathcal C^+;\V)$ is a Banach submanifold of $\mathcal M^{\univ}(\mathcal C^-,\mathcal C^+)$ for every $\V\in Y_a^{\reg}(\mathcal C^-,\mathcal C^+)$. Now the set
\begin{eqnarray*}
Y_a^{\reg}\coloneqq\bigcap_{(\mathcal C^-,\mathcal C^+)\in\CR^a(\YM)\times\CR^b(\E)}Y_a^{\reg}(\mathcal C^-,\mathcal C^+)
\end{eqnarray*}
is the intersection of finitely many residual subsets, hence residual in $Y_a$. For this set $Y_a^{\reg}$, the assertions of the theorem are satisfied.
\end{proof}

\begin{thm}[Transversality]\label{thm:linsurjetivederef} 
The horizontal differential $d\hat{\F}(w)$ of the map $\hat\F$ as in \eqref{eq:defnmaphatf} is surjective, for every pair $(\mathcal C^-,\mathcal C^+)\in\CR^a(\YM)\times\CR^b(\E)$ and every $w\in\mathcal M^{\univ}(\mathcal C^-,\mathcal C^+)$. 
\end{thm}

\begin{proof}
The theorem follows, combining Lemmata \ref{lem:transversstat} and \ref{lem:transversnonstat} below.
\end{proof}

To prove Theorem \ref{thm:linsurjetivederef} we distinguish the cases where $\mathcal C^+=\Phi(\mathcal C^-)$ (stationary case) or $\mathcal C^+\neq\Phi(\mathcal C^-)$ (non-stationary case). We are naturally led to this distinction by our definition of $a$-admissible perturbations which are assumed to be supported away from the critical manifolds.

\subsubsection*{Transversality at stationary flow lines}
Throughout we fix a pair $(\mathcal C^-,\mathcal C^+)\in\CR^a(\YM)\times\CR^b(\E)$ such that $\mathcal C^+=\Phi(\mathcal C^-)$ is satisfied. In this case, we show that transversality of the section $\F$ holds automatically (i.e.~for the perturbation $\V=0$).

\begin{prop}\label{prop:trivialkernel}
Let $[u]\in\mathcal M(\mathcal C^-,\mathcal C^+;\V)$ where $u\equiv(A,0,x)$ for some $A\in\mathcal C^-$ and $x=\Phi(A)$. Then $\ker\mathcal D_u^{\delta}$ is trivial.
\end{prop}

\begin{proof}
Let $(\alpha,\psi,\xi)\in\ker\mathcal D_u^{\delta}$ and consider the maps 
\begin{align*}
\ph^-\colon&\R^-\to\R,\quad s\mapsto\|(\alpha(s),\psi(s))\|_{L^2(\Sigma)}^2,\\
\ph^+\colon&\R^+\to\R,\quad s\mapsto\|\xi(s)\|_{L^2(S^1)}^2.
\end{align*}
As by assumption $\dot\zeta+\He_A\zeta=0$ is satisfied by $\zeta\coloneqq(\alpha,\psi)$ it follows that
\begin{eqnarray}\label{eq:firstderivph}
\dot\ph^-(s)=-2\langle\zeta,\He_A\zeta\rangle,\qquad\ddot\ph^-(s)=4\langle\He_A\zeta,\He_A\zeta\rangle\geq0.
\end{eqnarray}

The inequality in \eqref{eq:firstderivph} shows that $\ph^-$ is convex. Because $\lim_{s\to-\infty}\ph^-(s)=0$ it thus follows that $\ph^-$ vanishes identically or $\dot\ph^->0$. Assume by contradiction the second case. Then the first identity in \eqref{eq:firstderivph} shows that $\langle\zeta(0),\He_A\zeta(0)\rangle<0$ and from Proposition \ref{prop:infindecompequation} it follows that
\begin{eqnarray}\label{eq:inequllesszero}
\langle\xi_0,H_x\xi_0\rangle<0,
\end{eqnarray}
where we denote $\xi_0\coloneqq\xi(0)=\di\Phi(A)\zeta(0)$. Likewise, from the assumption that $\xi$ satisfies $\dot\xi+H_x\xi=0$ it follows that
\begin{eqnarray*}
\dot\ph^+(s)=-2\langle\xi,H_x\xi\rangle,\qquad\ddot\ph^+(s)=4\langle H_x\xi,H_x\xi\rangle\geq0,
\end{eqnarray*}
and so the map $\ph^+$ is convex. Because $\lim_{s\to\infty}\ph^+(s)=0$ it follows that $\dot\ph^+\leq0$ and hence in particular $\dot\ph^+(0)=-2\langle\xi_0,H_x\xi_0\rangle\leq0$. This contradicts \eqref{eq:inequllesszero} and shows that our assumption was wrong. Hence $\ph^-$ and therefore $\zeta$ vanish identically. Now $\xi_0=\di\Phi(A)\zeta(0)=0$, and convexity of $\ph^+$ shows that also $\xi$ vanishes identically. Hence $\ker\mathcal D_u^{\delta}$ is trivial, as claimed.
\end{proof}

\begin{lem}\label{lem:transversstat}
The horizontal differential $d\F([u])$ is surjective, for every $[u]\in\mathcal M(\mathcal C^-,\mathcal C^+;\V)$. 
\end{lem}

\begin{proof}
By applying a suitable gauge transformations we may assume that $[u]$ is represented by $u$ of the form $u=(A,0,x)$. Because $\mathcal C^+=\Phi(\mathcal C^-)$ by assumption, it follows from the gradient flow property that $u$ does not depend on $s$ and therefore satisfies the assumptions of Proposition \ref{prop:trivialkernel}. In this case it follows from \eqref{eq:finalindex} that $\ind d\F(u)=\dim\mathcal C^-$. We check that this number equals the dimension of $\ker d\F(u)$, which will imply the claim. First, the map $[\alpha_0]\mapsto(\alpha,\psi,\xi)$ where $(\alpha(s),\psi(s))\coloneqq(\alpha_0,0)$ for $s\in\R^-$ and $\xi(s)\coloneqq\di\Phi(A)\alpha_0$ for $s\in\R^+$ is an inclusion of $T_{[A]}\mathcal C^-$ into $\ker d\F(u)$. Second, this map is surjective because otherwise the kernel of the map $\mathcal D_u^{\delta}$ could not be trivial, in contradiction to Proposition \ref{prop:trivialkernel}. The claim follows.
\end{proof}

\subsubsection*{Transversality in the non-stationary case}

Let $(\mathcal C^-,\mathcal C^+)\in\CR^a(\YM)\times\CR^b(\E)$, where $b=4a/\pi$. We show surjectivity of the linearized operator in the case where $\mathcal C^+\neq\Phi(\mathcal C^-)$. Under this assumption, the following result holds true.

\begin{lem}\label{lem:transversnonstat}
The horizontal differential  $d\hat{\F}(w)$ is surjective, for every $w=[(A,\Psi,x,\V)]\in\mathcal M^{\univ}(\mathcal C^-,\mathcal C^+)$. 
\end{lem}

The proof is based on the following auxiliary result.

\begin{prop}\label{alternativesurj}
Let $w=[(A,\Psi,x,\V)]\in\mathcal M^{\univ}(\mathcal C^-,\mathcal C^+)$ and denote $\V=(\V^-,\V^+)$. Then the map $d\hat{\F}(w)$ is onto if one of the following two conditions is satisfied. (i) The linear operator 
\begin{multline*}
\hat{\mathcal D}_{(A,\V^-)}\colon\{(\alpha,\psi)\in\Z^{p,-}\mid(\alpha(0),\psi(0))=0\}\times Y_a^-\to\mathcal L^{p,-},\\
(\alpha,\psi,v^-)\mapsto\mathcal D_A(\alpha,\psi)+\nabla v^-(A)
\end{multline*}
is surjective. (ii) The linear operator 
\begin{eqnarray*}
\hat{\mathcal D}_{(x,\V^+)}\colon\{\xi\in\Z^{p,+}\mid\xi(0)=0\}\times Y_b^+\to\mathcal L^{p,+},\quad(\xi,v^+)\mapsto\mathcal D_x\xi+\nabla v^+(x)
\end{eqnarray*}
is surjective.
\end{prop}

\begin{proof}
Assume case (ii). Let $(\beta,\eta,\nu)\in\mathcal L^{p,-}\times\mathcal L^{p,+}\times L^2(S^1,\mathfrak g)$ be given. By Lemma \ref{DAsurjective} the equation $\hat{\mathcal D}_{(A,\V^-)}(\alpha,\psi,v^-)=\beta$ admits a solution (with e.g.~$v^-=0$). Assumption (ii) implies that the equation $\hat{\mathcal D}_{(x,\V^+)}(\xi,v^+)=\eta$ can be solved for arbitrary $\xi(0)$, in particular for $\xi(0)=x(0)(\nu+D\Phi_{A(0)}\alpha(0))$, cf.~\eqref{definitionD_u} regarding the notation. For this $\xi(0)$, the condition $N_A(\alpha,\xi)=\nu$ is satisfied. It follows that $(\hat{\mathcal D}_{(A,\V^-)},\hat{\mathcal D}_{(x,\V^+)})(\alpha,\psi,\xi,v^-,v^+)=(\beta,\eta,\nu)$ and hence $d\hat{\F}(w)$ is onto. Assuming case (i) we may argue analogously, using Lemma \ref{Dgsurjective} at the place of Lemma \ref{DAsurjective}.
\end{proof}

\begin{proof}{\bf{(Lemma \ref{lem:transversnonstat})}}
From our initial assumption that $\mathcal C^+\neq\Phi(\mathcal C^-)$ it follows that the gradient flow lines $(A,\Psi)$ or $x$ are not stationary. In this situation, the transversality results \cite[Theorem 7.1]{Swoboda1} and \cite[Proposition 7.5]{Web} apply (with minor modifications to the present situation where the linear operators $\hat{\mathcal D}_{(A,\V^-)}$ and 
$\hat{\mathcal D}_{(x,\V^+)}$ are defined on half-infinite intervals). In the first case this yields surjectivity of $\hat{\mathcal D}_{(A,\V^-)}$, in the second case surjectivity of $\hat{\mathcal D}_{(x,\V^+)}$. The claim now follows from Proposition \ref{alternativesurj}.
\end{proof}

\section{Compactness}
For a given pair $(\mathcal C^-,\mathcal C^+)\in\CR^a(\YM)\times\CR^b(\E)$ (where $a>0$ and $b=\frac{4a}{\pi}$), we aim to show compactness of the moduli spaces $\mathcal M(\mathcal C^-,\mathcal C^+)$ (as defined in \eqref{modulispace}) up to so-called \emph{convergence to broken trajectories}. Let us first introduce this notion, following the book by Schwarz \cite[Definition 2.34]{Schw}.

\begin{defn}\label{compbrokentrajectories}
\upshape
A subset $K\subseteq\mathcal M(\mathcal C^-,\mathcal C^+)$ is called \emph{compact up to broken trajectories of order} $\mu=(\mu^-,\mu^+)\in\mathbbm N_0^2$ if for any sequence $[u^{\nu}]=[(A^{\nu},\Psi^{\nu},x^{\nu})]$ in $K$ the following alternative holds. Either $[u^{\nu}]$ possesses a $C^{\infty}$ convergent subsequence, or there exist the following:
\begin{compactenum}[(i)]
\item
numbers $0\leq\lambda^{\pm}\leq\mu^{\pm}$ and critical manifolds
\begin{eqnarray*}
\mathcal C_0^-=\mathcal C^-,\ldots,\mathcal C_{\lambda^-}^-\subseteq\CR(\YM)\qquad\textrm{and}\qquad\mathcal C_0^+=\mathcal C^+,\ldots,\mathcal C_{\lambda^+}^+\subseteq\CR(\E);
\end{eqnarray*}
\item
for each $0\leq j\leq\lambda^--1$ a connecting trajectory $(A_j,\Psi_j)\in\hat{\mathcal M}(\hat{\mathcal C}_j^-,\hat{\mathcal C}_{j+1}^-)$, a sequence of gauge transformations $(g_{j,\nu})_{\nu\in\mathbbm N}\subseteq\G(\hat P)$ and a sequence of reparametrization times $(\tau_{j,\nu}^-)_{\nu\in\mathbbm N}\subseteq[0,\infty)$;
\item
a triple $(A^{\ast},\Psi^{\ast},x^{\ast})\in\hat{\mathcal M}(\hat{\mathcal C}_{\lambda^-}^-,\hat{\mathcal C}_{\lambda^+}^+)$ and a sequence of gauge transformations $(g_{\lambda^-,\nu})_{\nu\in\mathbbm N}\subseteq\G(\hat P)$;
\item
for each $1\leq j\leq\lambda^+$ a connecting trajectory $x_j\in\hat{\mathcal M}(\hat{\mathcal C}_j^+,\hat{\mathcal C}_{j-1}^+)$ and a sequence of reparametrization times $(\tau_{j,\nu}^+)_{\nu\in\mathbbm N}\subseteq[0,\infty)$, with the following significance:
\end{compactenum}
There exists a subsequence (again labeled by $\nu$) such that, as $\nu\to\infty$, 
\begin{align*}
&g_{j,\nu}^{\ast}(A^{\nu}(\,\cdot\,-\tau_{j,\nu}^-),\Psi^{\nu}(\,\cdot\,-\tau_{j,\nu}^-))\to(A_j,\Psi_j)\qquad\textrm{for every}\;0\leq j\leq\lambda^--1,\\
&x^{\nu}(\,\cdot\,+\tau_{j,\nu}^+)\to x_j\qquad\textrm{for every}\;0\leq j\leq\lambda^+-1,\\
&g_{\lambda^-,\nu}^{\ast}(A^{\nu},\Psi^{\nu})\to(A^{\ast},\Psi^{\ast}),\qquad x^{\nu}\to x^{\ast}
\end{align*}
in $C^{\infty}$ on all compact domains $I\times\Sigma$, respectively $I\times S^1$, where $I\subseteq\R^-$ (respectively $I\subseteq\R^+$) is a compact interval.  
\end{defn}

Here the notation $\hat{\mathcal M}(\hat{\mathcal C}_j^{\pm},\hat{\mathcal C}_{j\mp1}^{\pm})$ refers to the moduli spaces of connecting trajectories for the gradient flows of $\YMV$, respectively of $\EV$ as introduced in \cite{Swoboda1} and \cite{Web}. As we show next, the moduli space $\mathcal M(\mathcal C^-,\mathcal C^+)$ is compact in the sense of Definition \ref{compbrokentrajectories}.

\begin{thm}[Compactness of moduli spaces]\label{Compactnessmodulispaces}
For every pair $(\mathcal C^-,\mathcal C^+)\in\CR(\YM)\times\CR(\E)$, the moduli space $\mathcal M(\mathcal C^-,\mathcal C^+)$ is empty or compact up to convergence to broken trajectories of order $\mu=(\mu^-,\mu^+)$, where 
\begin{eqnarray}\label{eq:orderbrokentraj}
\mu^-+\mu^+=\ind A^-+\dim\mathcal C^--\ind x^+.
\end{eqnarray}
(The integers $\ind A^-$ and $\ind x^+$ denote the number of negative eigenvalues of $\mathcal H_{A^-}$, respectively of $H_{x^+}$ (for $A^-\in\hat{\mathcal C}^-$ and $x^+\in\hat{\mathcal C}^+$).
\end{thm}

To prove the theorem we need the following two lemmata, the first one being due to the author \cite{Swoboda1} and the second one due to Weber \cite{Web}.

\begin{lem}\label{lem:compactnessYM}
Let $\mathbbm A^{\nu}=A^{\nu}+\Psi^{\nu}\,ds$, $\nu\in\mathbbm N$, be a sequence of solutions of the perturbed Yang--Mills gradient flow equation \eqref{EYF} on $\R^-\times\Sigma$. Assume there exists a critical manifold $\hat{\mathcal C}^-\in\hat{\mathcal{CR}}(\YM)$ such that $\mathbbm A^{\nu}(s)$ converges to $\hat{\mathcal C}^-$ as $s\to-\infty$, for every $\nu\in\mathbbm N$. Then there exists a sequence $g^{\nu}\in\mathcal G(\hat P)$ of gauge transformations such that a subsequence of the gauge transformed sequence $(g^{\nu})^{\ast}\mathbbm A^{\nu}$ converges uniformly on compact sets $I\times\Sigma$ to a solution $\mathbbm A^{\ast}$ of \eqref{EYF}. 
\end{lem}

\begin{proof}
For a proof we refer to \cite[Theorem 6.1]{Swoboda1}.
\end{proof}

\begin{lem}\label{lem:compactnessLoop}
Let $x^{\nu}$, $\nu\in\mathbbm N$, be a sequence of solutions of the gradient flow equation \eqref{EEF} on the interval $\R^+$. Assume the uniform energy bound
\begin{eqnarray}\label{eq:unifboundenergy}
\sup_{s\in\R^+}\E^{\V^+}(x^{\nu}(s))\leq C
\end{eqnarray}
holds for some constant $C$ and all $\nu\in\mathbbm N$. Then there exists a solution $x^{\ast}$ of \eqref{EEF} on $\R^+\times S^1$ such that a subsequence of $x^{\nu}$ converges uniformly to $x^{\ast}$ on compact sets $I\times S^1$. 
\end{lem}

\begin{proof}
For a proof we refer to \cite[Proposition 4.14]{Web}.
\end{proof}

\begin{proof}{\bf{(Theorem \ref{Compactnessmodulispaces})}}
Let $u^{\nu}=(A^{\nu},\Psi^{\nu},x^{\nu})$, $\nu\in\mathbbm N$, be a sequence in $\hat{\mathcal M}(\hat{\mathcal C}^-,\hat{\mathcal C}^+)$. The sequence $x^{\nu}$ satisfies condition \eqref{eq:unifboundenergy}. Namely, thanks to the energy inequality \eqref{decompinequality} it follows for a constant $C(\mathcal C^-)$ and all $\nu$ that
\begin{eqnarray*}
\sup_{s\in\R^+}\E^{\V^+}(x^{\nu}(s))=\E^{\V^+}(x^{\nu}(0))\leq\frac{4}{\pi}\YM^{\V^-}(A^{\nu}(0))\leq C(\mathcal C^-).
\end{eqnarray*}
The last inequality follows from the assumption that $A^{\nu}(s)$ converges to $\hat{\mathcal C}^-$ as $s\to-\infty$. Hence Lemmata \ref{lem:compactnessYM} and \ref{lem:compactnessLoop} apply and show that there exists a subsequence of $u^{\nu}$ (which we still label by $\nu$) and a sequence $g^{\nu}$ of gauge transformations such that $(g^{\nu})^{\ast}(A^{\nu},\Psi^{\nu})$ and $x^{\nu}$ converge uniformly on compact sets $I\times\Sigma$, respectively $I\times S^1$. Furthermore, the limit as $\nu\to\infty$ of $(g^{\nu})^{\ast}|_{I\times\Sigma}(A^{\nu}|_{I\times\Sigma},\Psi^{\nu}|_{I\times\Sigma})$ (respectively of $x^{\nu}|_{I\times S^1}$) is the restriction of a solution of \eqref{EYF} or \eqref{EEF} of finite energy at most $C(\mathcal C^-)$. Then the exponential decay results \cite[Theorem 4.1]{Swoboda1} and \cite[Theorem 1.8]{Web} imply that every such finite energy solution is contained in some moduli space $\hat{\mathcal M}(\hat{\mathcal C}_0^-,\hat{\mathcal C}_1^-)$, $\hat{\mathcal M}(\hat{\mathcal C}_0^+,\hat{\mathcal C}_1^+)$, or $\hat{\mathcal M}(\hat{\mathcal C}_0^-,\hat{\mathcal C}_0^+)$ for some $\hat{\mathcal C}_j^-\in\widehat{\CR}^a(\YM)$, respectively $\hat{\mathcal C}_j ^+\in\widehat{\CR}^b(\E)$, where $j=0,1$ and $a=\YM(\hat{\mathcal C}^-)$, $b=4a/\pi$. (Here we use the notation of Section \ref{sect:critmanifolds}). Convergence after reparametrization as required in Definition \ref{compbrokentrajectories} and the relation \eqref{eq:orderbrokentraj} then follow from standard arguments as in \cite[Proposition 2.35]{Schw}.
\end{proof}

\section{Chain isomorphism of Morse complexes}

\subsection{The chain map}\label{sec:chainmap}
\setcounter{footnote}{0}
Let $a\geq0$ be a regular value of $\YM$ and set $b\coloneqq4a/\pi$. Throughout this section we fix an admissible perturbation $\V=(\V^-,\V^+)\in Y_a^{\reg}$ (with $Y_a^{\reg}\subseteq Y_a$ as in Theorem \ref{thm:transversality}) such that the conditions of Theorem \ref{thmdecompinequality} are satisfied. Let $h\colon\crit^a(\YM)/\mathcal G_0(P)\to\R$ be a smooth Morse--Smale function, i.e.~a smooth Morse function such that for all $x,y\in\crit(h)$ the stable and unstable manifolds $W_h^s(x)$ and $W_h^u(y)$ of $h$ intersect transversally. We let 
\begin{eqnarray*}
CM_{\ast}^{a,-}\coloneqq CM_{\ast}^a\big(\A(P)/\mathcal G_0(P),\V^-,h\big),\qquad CM_{\ast}^{b,+}\coloneqq CM_{\ast}^b\big(\Lambda G/G,\V^+,h\big)
\end{eqnarray*}
denote the Morse--Bott complexes as in Section \ref{sect:Morsehomologies}. The corrsponding Morse homology groups will for short be denoted by
\begin{eqnarray*}
HM_{\ast}^{a,-}\coloneqq HM_{\ast}^a\big(\A(P)/\mathcal G_0(P),\V^-,h\big),\qquad HM_{\ast}^{b,+}\coloneqq HM_{\ast}^b\big(\Lambda G/G,\V^+,h\big).
\end{eqnarray*}
 
The construction of the moduli space $\mathcal M(\mathcal C^-,\mathcal C^+)$ in Section \ref{sect:Hybridmodspace} gives rise to a chain map $\Theta\colon CM_{\ast}^{a,-}\to CM_{\ast}^{b,+}$ as we shall describe next.

\begin{defn}\label{hybridflowline}
\upshape
Fix critical points $x^-,x^+\in\crit(h)$ and integers $m^-,m^+\geq1$. A \emph{hybrid flow line from} $x^-$ \emph{to} $x^+$ \emph{with} $m^-$ \emph{upper and} $m^+$ \emph{lower cascades} is a tuple 
\begin{multline*}
(\underline{x}^-,\underline{x}^0,\underline{x}^+,T^-,T^+)\\
=((x_j^-)_{j=1,\ldots,m^-},\underline{x}^0,(x_j^+)_{j=1,\ldots,m^+},(t_j^-)_{j=1,\ldots,m^-},(t_j^+)_{j=0,\ldots,m^+-1}),
\end{multline*}
where for each $j$, $x_j^-\colon\R^-\to\A(P)/\G_0(P)$ is a nonconstant solution of the Yang--Mills gradient flow equation \eqref{EYF}, $x_j^+\colon\R^+\to\Lambda G/G$ is a nonconstant solution of the loop group gradient flow equation \eqref{EEF}, $t_j^{\pm}\in\mathbbm R^+$, and the following conditions are satisfied.
\begin{compactenum}[(i)]
\item
For each $1\leq j\leq m^{\pm}-1$ there exists a solution $y_j^{\pm}\in C^{\infty}(\mathbbm R,\crit(h))$ of the gradient flow equation $\dot y_j^{\pm}=-\nabla h(y_j^{\pm})$ such that $\lim_{s\to\infty}x_j^{\pm}(s)=y_j^{\pm}(0)$ and $\lim_{s\to-\infty}x_{j+1}^{\pm}(s)=y_j^{\pm}(t_j)$.
\item
There exist $p^-\in W_h^u(x^-)$ and $p^+\in W_h^s(x^+)$ such that $\lim_{s\to-\infty}x_1^-(s)=p^-$ and $\lim_{s\to\infty}x_m^+(s)=p^+$.
\item
There exist $\mathcal C^-\in\mathcal{CR}(\YM)$ and $\mathcal C^+\in\mathcal{CR}(\E)$ such that $\underline{x}^0=[(u^-,u^+)]\in\mathcal M(\mathcal C^-,\mathcal C^+)$. Furthermore, there exist solutions $y_{m^-}^-$ and $y_0^+$ of the gradient flow equations $\dot y_{m^-}^-=-\nabla h(y_{m^-}^-)$ and $\dot y_0^+=-\nabla h(y_0^+)$ satisfying the conditions 
\begin{align*}
&\lim_{s\to\infty}x_{m^-}^-(s)=y_{m^-}^-(0),\qquad\lim_{s\to-\infty}u^-(s)=y_{m^-}^-(t_{m^-}^-),\\
&\lim_{s\to\infty}u^+(s)=y_0^+(0),\qquad\lim_{s\to-\infty}x_1^+(s)=y_0^+(t_0^+).
\end{align*}
\end{compactenum}
A {\emph{hybrid flow line with $m^-=0$ upper cascades and $m^+\geq1$ lower cascades}} is a tuple $(\underline{x}^0,\underline{x}^+,T^+)=(\underline{x}^0,(x_j^+)_{j=1,\ldots,m^+},(t_j^+)_{j=0,\ldots,m^+-1})$ satisfying conditions (i-iii) above with the following adjustment in (ii). Here we require the existence of $p^-\in W_h^u(x^-)$ such that $\lim_{s\to-\infty}u^-(s)=p^-$. Conditions involving flow lines $x_j^-$ and times $t_j^-$ are empty in this case. A {\emph{hybrid flow line with $m^-\geq0$ upper cascades and $m^+=0$ lower cascades}} is defined analogously. Conditions involving flow lines $x_j^+$ and times $t_j^+$ are then empty.
\end{defn}

\begin{figure}\label{bild1}   
\includegraphics[scale=0.46]{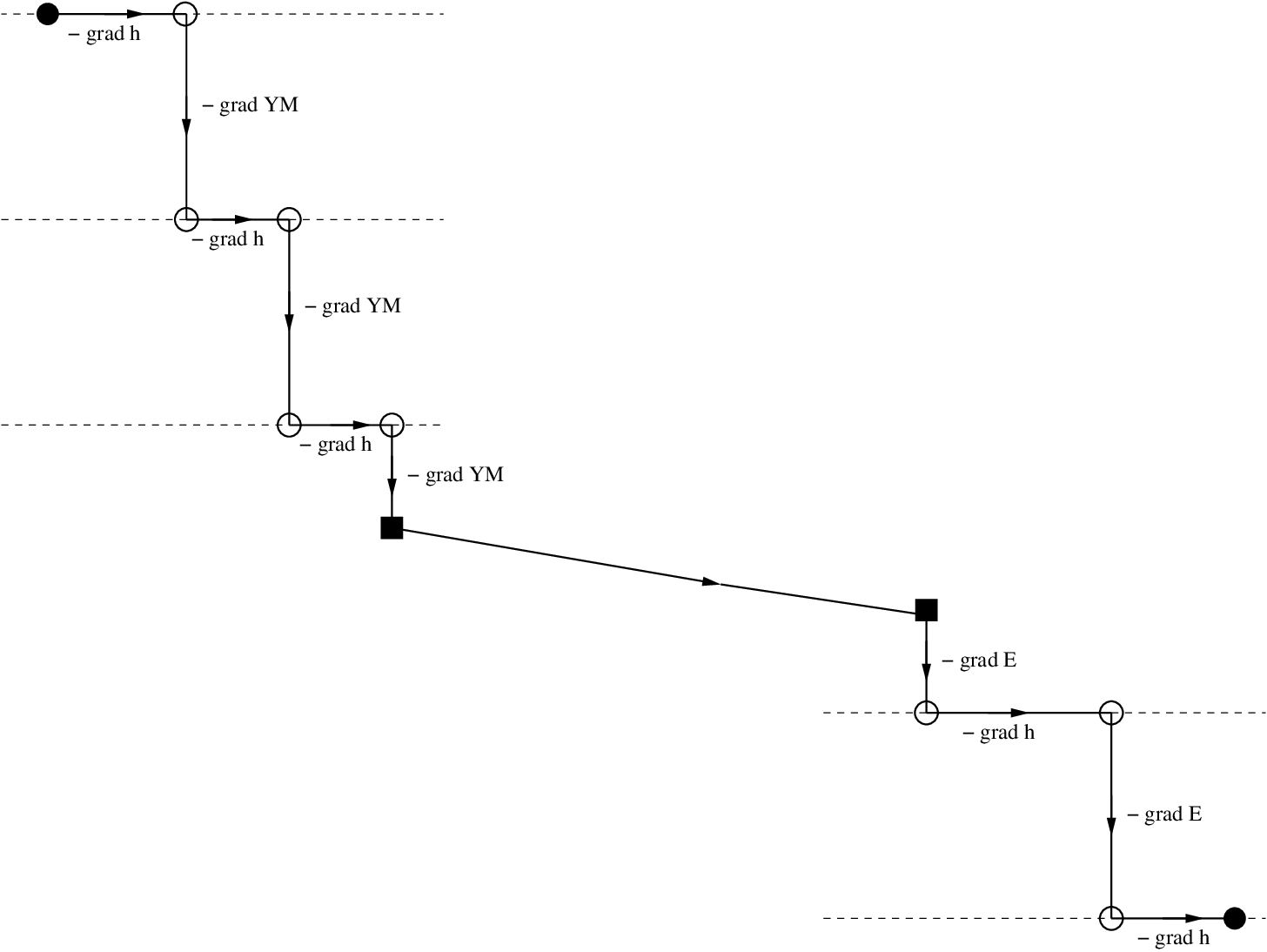}
\caption{A hybrid flow line with $2$ upper and $1$ lower cascades.} 
\end{figure}

\begin{figure}[h]\label{bild1}
\begin{center}
\end{center}
\caption{A hybrid flow line with $2$ upper and $1$ lower cascades.} 
\end{figure}

For $(m^-,m^+)\in\mathbbm N_0^2$ we denote by $\hat{\mathcal M}_{(m^-,m^+)}^{\hybr}(x^-,x^+)$ the set of hybrid flow lines from $x^-$ to $x^+$ with $m^-$ upper and $m^+$ lower cascades. The group $\mathbbm R^{m^-}\times\mathbbm R^{m^+}$ acts on tuples $(\underline{x}^-, \underline{x}^+)$ by time-shifts. The resulting quotient space is called the {\emph{moduli space of hybrid flow lines from $x^-$ to $x^+$ with $m^-$ upper and $m^+$ lower cascades}}. We call 
\begin{eqnarray*}
\mathcal M^{\hybr}(x^-,x^+)\coloneqq\bigcup_{(m^-,m^+)\in\mathbbm N_0^2}\mathcal M_{(m^-,m^+)}^{\hybr}(x^-,x^+)
\end{eqnarray*}
the moduli space of {\emph{hybrid flow lines with cascades from}} $x^-$ {\emph{to}} $x^+$. 
  
\begin{lem}
The dimension of $\mathcal M^{\hybr}(x^-,x^+)$ is given by the formula 
\begin{eqnarray*} 
\dim\mathcal M^{\hybr}(x^-,x^+)=\Ind(x^-)-\Ind(x^+).
\end{eqnarray*}
(The nonnegative integers $\Ind(x^{\pm})$ denote the sum of the Morse indices of $x^-$ as a critical point of $\YM$ and of $h$, respectively the sum of the Morse indices of $x^+$ as a critical point of $\E$ and of $h$).
\end{lem}

\begin{proof}
The assertion follows from the following formula for the dimension of $\hat{\mathcal M}_{(m^-,m^+)}^{\hybr}(x^-,x^+)$. Namely, by the result \cite[Corollary C.15]{Frauenfelder1} and the index formula \eqref{eq:FredholmindexdF} it follows that
\begin{eqnarray*}
\dim\hat{\mathcal M}_{(m^-,m^+)}^{\hybr}(x^-,x^+)=\Ind(x^-)-\Ind(x^+)+m-1,
\end{eqnarray*}
where $m$ denotes the total number of cascades of the elements of $\hat{\mathcal M}_{(m^-,m^+)}^{\hybr}(x^-,x^+)$. Adapting the arguments of \cite[Corollary C.15]{Frauenfelder1} to the present situation it follows that $m=m^-+m^++1$ (we here have to take into account the additional cascade coming from the configuration $\underline{x}_0$). The asserted formula for the quotient space modulo the action of $\mathbbm R^{m^-}\times\mathbbm R^{m^+}$ then follows.
\end{proof}

We obtain a chain map between the Morse--Bott complexes $CM_{\ast}^{a,-}$ and $CM_{\ast}^{b,+}$ by counting the number elements of $\mathcal M^{\hybr}(x^-,x^+)$ for appropriate pairs $(x^-,x^+)$ of generators.

\begin{defn}\upshape
For a pair $(x^-,x^+)\in\crit(h)\times\crit(h)$ with $\Ind(x^-)=k$ let
\begin{eqnarray*}
\Theta_k(x^-)\coloneqq\sum_{x^+\in\crit(h)\atop\Ind(x^+)=k}\#\,\mathcal M^{\hybr}(x^-,x^+)\cdot x^+,
\end{eqnarray*}
where $\#\,\mathcal M^{\hybr}(x^-,x^+)$ denotes the number (counted modulo $2$) of elements of $\mathcal M^{\hybr}(x^-,x^+)$. We define the homomorphism $\Theta_k\colon CM_k^{a,-}\to CM_k^{b,+}$ of abelian groups accordingly by linear continuation, and set $\Theta\coloneqq(\Theta_k)_{k\in\mathbbm N_0}$.
\end{defn}

\begin{thm}[Chain map]\label{theoremchainmap}
The map $\Theta$ is a chain homomorphism between the Morse complexes $CM_{\ast}^{a,-}$ and $CM_{\ast}^{b,+}$. Thus for each $k\in\mathbbm N_0$ it holds $\Theta_k\circ\partial_{k+1}^{\,\YM}=\partial_{k+1}^{\,\E}\circ\Theta_{k+1}\colon CM_{k+1}^{a,-}\to CM_k^{b,+}$.
\end{thm}

\begin{proof} 
From the compactness Theorem \ref{Compactnessmodulispaces} it follows that the number of elements of $\mathcal M^{\hybr}(x^-,x^+)$ is finite and hence the homomorphism $\Theta$ is well-defined. The proof is then completed by standard arguments as e.g.~carried out in \cite{AbSchwarz}. 
\end{proof}

\subsection{Proof of the main theorem}\setcounter{footnote}{0}

The aim of this final section is to prove Theorem \ref{thm:mainresult}. Thanks to Theorem \ref{theoremchainmap}, the map $\Theta$ induces for each $k\in\mathbbm N_0$ a homomorphism
$[\Theta_k]\colon HM_k^{a,-}\to HM_k^{b,+}$ of abelian groups. It remains to show that these homomorphisms are in fact isomorphisms. In our proof we follow closely the line of argument employed by Abbondandolo and Schwarz in \cite{AbSchwarz}.

\begin{proof}{\bf{(Theorem \ref{thm:mainresult})}}
It suffices to prove that each of the chain homomorphisms $\Theta_k$ is in fact an isomorphism and hence induces an isomorphism in homology. Let $k\in\mathbbm N_0$ and fix a set $\underline{p}\coloneqq(p_1,\ldots,p_m)$ of generators of $CM_k^{a,-}$. We order the entries $p_j$ of the tuple $\underline{p}$ such that the following two conditions are met. First, we require that
\begin{eqnarray}\label{ordering1YM}
i\leq j\qquad\Longrightarrow\qquad\YM^{\V^-}(p_i)\leq\YM^{\V^-}(p_j)
\end{eqnarray}
holds for all $1\leq i,j\leq m$. Secondly, in case where $p_i$ and $p_j$ lie on the same critical manifold $\mathcal C^-\in\CR(\YM)$ we choose our ordering such that 
\begin{eqnarray}\label{ordering2YM}
i\leq j\qquad\Longrightarrow\qquad h(p_i)\leq h(p_j).
\end{eqnarray}
Set $q_j\coloneqq\Phi(p_j)$ and define $\underline{q}\coloneqq(q_1,\ldots,q_m)$. The $m$-tuple $\underline{q}$ generates $CM_k^{b,+}$ which follows from the facts that $\Phi$ induces a bijection $\CR^a(\YM)\to\CR^b(\E)$ which by Theorem \ref{critcorr} preserves the Morse indices, and that on both sets  $\CR^a(\YM)$ and $\CR^b(\E)$ we use the same Morse function $h$. By Theorem \ref{thmdecompinequality} we have for each $1\leq j\leq m$ the identity
\begin{eqnarray}\label{decompequationfinal}
\mathcal{YM}^{\V^-}(p_j)=\frac{\pi}{4}\mathcal E^{\V^+}(q_j), 
\end{eqnarray}
which by our choice of the ordering of $\underline{p}$ implies that the tuple $\underline{q}$ is ordered by non-decreasing $\E^{\V^+}$ action. Furthermore, if $\E^{\V^+}(q_i)=\E^{\V^+}(q_j)$ for some $i\leq j$ then either $q_i$ and $q_j$ lie on different critical manifolds in $\CR^b(\E)$ or otherwise $h(q_i)\leq h(q_j)$. Let us represent the homomorphism $\Theta_k$ with respect to the ordered bases $\underline{p}$ and $\underline{q}$ by the matrix $(\Theta_{ij}^k)_{1\leq i,j\leq m}\in\mathbbm Z_2^{m\times m}$. The following two observations are now crucial. Both are a consequence of the energy inequality 
\begin{eqnarray}\label{decompinequalityfinal}
\mathcal{YM}^{\V^-}(A)\geq\frac{\pi}{4}\mathcal E^{\V^+}(\Phi(A))
\end{eqnarray}
for all $A\in\A(P)$ as in Theorem \ref{thmdecompinequality}. First, $\Theta_{ii}^k=1$ for all $1\leq i\leq m$ because the moduli space $\mathcal M^{\hybr}(p_i,q_i)$ consists of precisely one point. It is represented by the hybrid flow line with $(m^-,m^+)=(0,0)$ upper and lower cascades and configuration $\underline{x}^0=(p_i,q_i)=[(u^-,u^+)]$ with stationary flow lines $u^-$ of $\YMV$, respectively $u^+$ of $\EV$. Note that $(m^-,m^+)\neq(0,0)$ is not possible in this case as this would contradict \eqref{decompequationfinal}. Secondly, if $i>j$ then $\Theta_{ij}^k=0$ because in this case $\mathcal M^{\hybr}(p_j,q_i)=\emptyset$. Assume by contradiction that $\mathcal M^{\hybr}(p_j,q_i)$ contains at least one element. Let $p_j\in\mathcal C^-$ and $q_i\in\mathcal C^+$ for critical manifolds $\mathcal C^{\pm}$. The gradient flow property and \eqref{decompinequalityfinal} imply that $\YM^{\V^-}(p_j)\geq\frac{\pi}{4}\E^{\V^+}(q_i)$. This inequality must in fact be an equality because otherwise there would be a contradiction to condition \eqref{ordering1YM} and identity \eqref{decompequationfinal}. Hence $\mathcal C^+=\Phi(\mathcal C^-)$, and the gradient flow property of the function $h$ implies that $h(q_j)=h(\Phi(p_j))\geq h(q_i)$. With $i>j$, condition \eqref{ordering2YM} and again identity \eqref{decompequationfinal} show that this can only be the case if $h(q_j)=h(q_i)$. It follows that $q_j=q_i$ because $h$ is monotone decreasing. Thus $i=j$, which is a contradiction. These two observations imply that the matrix $(\Theta_{ij}^k)_{1\leq i,j\leq m}$ takes the form
\begin{eqnarray*}
(\Theta_{ij}^k)_{1\leq i,j\leq m}=\left(\begin{array}{ccccc}1&\ast&\cdots&\cdots&\ast\\0&1&\ddots&&\vdots\\\vdots&\ddots&\ddots&\ddots&\vdots\\\vdots&&\ddots&1&\ast\\0&\cdots&\cdots&0&1\end{array}\right)\in\mathbbm Z_2^{m\times m},
\end{eqnarray*}
and thus is invertible in $\mathbbm Z_2^{m\times m}$. This completes the proof of the theorem.
\end{proof}

\appendix

\section{Holonomy map and critical manifolds}\label{holmapappendix}

We let $\Sigma=S^2\subseteq\mathbbm R^3$ be the unit sphere, endowed with the standard round metric induced from the ambient euclidian space $\mathbbm R^3$. Fix the pair $z^{\pm}=(0,0,\pm1)$ of antipodal points and set $D^{\pm}\coloneqq\Sigma\setminus\{z^{\mp}\}$. We parametrize the hypersurfaces $D^{\pm}\subseteq\R^3$ by the maps
\begin{eqnarray*}
u^{\pm}\colon[0,\pi)\times[0,2\pi)\to D^{\pm},\quad(r,t)\mapsto\big(\pm\cos(t)\sin(r),\pm\sin(t)\sin(r),\pm\cos(r)\big).
\end{eqnarray*}

The inverses of $u^{\pm}$ give rise to coordinate charts on $\Sigma$. Set $\lambda\colon[0,\pi)\to\mathbbm R$, $r\mapsto\sin(r)$. The Riemannian metric in these coordinates is $g^{\pm}=dr^2+\lambda^2\,dt^2$ and the volume form is $\dvol(D^{\pm})=\lambda\,dr\wedge dt$. The Hodge star operator acts on differential forms on $D^{\pm}$ as
\begin{eqnarray*}
\ast1=\lambda\,dr\wedge dt,\qquad\ast\,dr\wedge dt=\lambda^{-1},\qquad\ast\,dr=\lambda\,dt,\qquad\ast\,dt=-\lambda^{-1}\,dr.
\end{eqnarray*}

We define the family $\gamma_t$ ($0\leq t\leq2\pi$) of paths between $z^+$ and $z^-$ as
\begin{eqnarray*} 
\gamma_t\colon[0,\pi]\to\Sigma,\qquad\gamma_t(r)\coloneqq(\cos(t)\sin(r),\sin(t)\sin(r),\cos(r)).
\end{eqnarray*}

We fix two points $p^{\pm}\in P_{z^{\pm}}$ in the fibre $P_{z^{\pm}}\subseteq P$ above $z^{\pm}$. For a connection $A\in\A(P)$, let $\bar\gamma_t^A$ denote the horizontal lift of the path $\gamma_t$ with respect to $A$, starting at $p^+$.

\begin{defn}\label{holonomymap}
\upshape
For the parameter $0\leq t\leq2\pi$ we define $g_A(t)\in G$ by the condition $\bar\gamma_t^A(\pi)=p^-\cdot g_A(t)$. Therefore $t\mapsto g_A(t)$ is a loop in $\Lambda G$. Now we define the map
\begin{eqnarray}\label{mapPhi}
\Phi\colon\A(P)\rightarrow\Omega G,\quad A\mapsto x_A\coloneqq g_A^{-1}(0)g_A.
\end{eqnarray}
The map $\Phi$ is called the \emph{holonomy map} of the principal $G$-bundle $P$.
\end{defn}
 
Note that $\Phi$ is independent of the choice of $p^-$, while replacing $p^+$ by $p^+\cdot h$ for some $h\in G$ results in the conjugate holonomy map $h\Phi h^{-1}$. Similarly, replacing $A$ by $g^{\ast}A$ for a gauge transformation $g\in\mathcal G(P)$ with $g(p^+)=h\in G$ yields $\Phi(g^{\ast}A)=h\Phi(A)h^{-1}$. Thus for the subgroup of gauge transformations $\G_0(P)$ based at $z^+$, i.e.~for
\begin{eqnarray*} 
\mathcal G_0(P)\coloneqq\{g\in\mathcal G(P)\mid g(p^+)=\mathbbm 1\},
\end{eqnarray*}
it follows that $\Phi$ descends to a $G$-equivariant map $\Phi\colon\A(P)/\mathcal G_0(P)\to\Omega G$, again denoted $\Phi$ and named holonomy map. The next theorem gives an explicit description of the set of Yang--Mills connections on the principal $G$-bundle $P$. Recall that the set of isomorphism classes of principal $G$-bundles is in bijection with the elements of the fundamental group $\pi_1(G)$, as any such bundle $P$ is determined up to isomorphism by the homotopy class of the transition map $D^+\cap D^-\to G$ of a trivialization of $P$ over the open sets $D^{\pm}$.

\begin{thm}[Correspondence between critical points]\label{critcorr}  
Let $P$ be a principal $G$-bundle over $\Sigma$ of topological type $\alpha\in\pi_1(G)$. Then the map $\Phi\colon A\mapsto x_A$ induces a bijection between the set of gauge equivalence classes of Yang--Mills connections on $P$ and the set of conjugacy classes of closed, based geodesics $x$ on $G$ of homotopy class $[x]=\alpha$.
\end{thm}

\begin{proof}
For a proof we refer to \cite[Theorem 2.1]{FriedrichHabermann}.
\end{proof}

\begin{thm}{\label{ThmFriedrichHabermann2}}
Let $[A]\in\A(P)/\G_0(P)$ be a based gauge equivalence class of Yang--Mills connections, and let $x_A=\Phi(A)\in\Omega G$ denote the corresponding closed geodesic. Then the Hessians of $\YM$ at $[A]$ and of $\E$ at $x_A$ have the same index and nullity.
\end{thm}

\begin{proof}
For a proof we refer to \cite[Theorem 2.2]{FriedrichHabermann}.
\end{proof}

Let $i^{\pm}\colon D^{\pm}\to\Sigma$ denote the inclusion maps and $P^{\pm}\coloneqq(i^{\pm})^{\ast}P$ the corresponding pull-back bundles. Because $D^{\pm}\subseteq\Sigma$ is contractible, the bundles $P^{\pm}$ are trivial and thus admit sections over $D^{\pm}$. Given $A\in\A(P)$ we can choose sections $s^{\pm}\colon D^{\pm}\to P^{\pm}$ in such a way that the pull-back of the connections $(i^{\pm})^{\ast}A\in\A(P^{\pm})$ under $s^{\pm}$ are of the form $u^{\pm}\,dt$ for maps $u^{\pm}\in C^{\infty}(D,\mathfrak g)$ satisfying
\begin{eqnarray}\label{eq:compatcondu}
u^-(r,t)=x_A(t)u^+(\pi-r,t)x_A^{-1}(t)-\partial_tx_A(t)x_A^{-1}(t)
\end{eqnarray}
for $0<r<\pi$. Because the connections $(i^{\pm})^{\ast}A$ are well-defined near $r=0$ it follows that $\lim_{r\to0}u^{\pm}(r,t)=0$. Following Gravesen \cite{Gravesen} we set
\begin{eqnarray*}
\ell(r)\coloneqq\frac{1}{2}(1-\cos(r))\qquad\textrm{and}\qquad\xi_A\coloneqq x_A^{-1}\partial_tx_A\in\Omega\mathfrak g,
\end{eqnarray*}
and split the map $u^+$ as
\begin{eqnarray}\label{decomposition}
u^+(r,t)=\ell(r)\xi_A(t)+m_A(r,t).
\end{eqnarray}

($m_A\in C^{\infty}(D^+,\mathfrak g)$ being defined through this equation). Because $u^+(r,t)\to0$ as $r\to0$ this definition implies that also $\lim_{r\to0}m_A(r,t)=0$. Inserting \eqref{decomposition} into \eqref{eq:compatcondu} yields
\begin{eqnarray*}
u^-(r,t)=\ell(\pi-r)x_A(t)\xi_A(t)x_A^{-1}(t)+x_A^{-1}(t)m_A(\pi-r,t)x_A(t)-\partial_tx_A(t)x_A^{-1}(t).
\end{eqnarray*}
 
Because $\ell(\pi)=1$ and $u^-(r,t)\to0$ as $r\to0$ it follows that $\lim_{r\to\pi}m(r,t)=0$. The following energy identity is due to Gravesen \cite[Section 2]{Gravesen}. There is also a generalization to higher genus surfaces, cf.~Davies \cite[Section 4.2]{Davies}.

\begin{lem}[Energy identity]\label{decomplemma}
For every $A\in\A(P)$, the identity
\begin{eqnarray}\label{decompequation}
\mathcal{YM}(A)=\frac{\pi}{4}\E(x_A)+\frac{1}{2}\|\lambda^{-1}\partial_rm_A(r,t)\|_{L^2(\Sigma,\dvol(\Sigma))}^2.
\end{eqnarray}
is satisfied.
\end{lem}

\begin{proof}
We identify $A\in\A(P)$ locally on $D^{\pm}$ with $1$-forms $u^{\pm}\in\Omega^1(D^{\pm},\mathfrak g)$ as in \eqref{eq:compatcondu} and make use of the decomposition \eqref{decomposition}. It follows that the curvature $F_A$ is identified on $D^{\pm}$ with the $2$-form $\partial_ru^{\pm}\,dr\wedge dt$. Recall also the formula $\ast(dr\wedge dt)=\lambda^{-1}$ for the Hodge star operator on $2$-forms. In addition, we use that $\partial_r\ell(r)=\frac{1}{2}\lambda(r)$. Hence it follows that
\begin{eqnarray*}
\lefteqn{\mathcal{YM}(A)=\frac{1}{2}\int_{\Sigma}\langle F_A\wedge\ast F_A\rangle}\\
&=&\frac{1}{2}\int_{D^+}\lambda^{-1}\langle\partial_ru^+,\partial_ru^+\rangle\,dr\wedge dt\\
&=&\frac{1}{2}\int_{D^+}\lambda^{-1}(r)\left\langle\frac{1}{2}\lambda(r)\xi_A(t)+\partial_rm_A(r,t),\frac{1}{2}\lambda(r)\xi_A(t)+\partial_rm_A(r,t)\right\rangle\,dr\wedge dt\\
&=&\frac{1}{2}\int_{D^+}\left\langle\frac{1}{2}\xi_A(t)+\lambda^{-1}(r)\partial_rm_A(r,t),\frac{1}{2}\xi_A(t)+\lambda^{-1}(r)\partial_rm_A(r,t)\right\rangle\,\lambda(r)dr\wedge dt\\
&=&\frac{\pi}{8}\int_0^{2\pi}\langle\xi_A(t),\xi_A(t)\rangle\,dt+\frac{1}{2}\int_{D^+}\big|\lambda^{-1}(r)\partial_rm_A(r,t)\big|^2\,\lambda(r)dr\wedge dt\\
&&+\frac{1}{2}\int_{D^+}\partial_r\langle\xi_A(t),m_A(r,t)\rangle\,dr\wedge dt\\
&=&\frac{\pi}{4}\E(x_A)+\frac{1}{2}\|\lambda^{-1}\partial_rm_A(r,t)\|_{L^2(D^+,\dvol(D^+))}^2.
\end{eqnarray*}
The term in the second but last line vanishes as follows from the above stated property $\lim_{r\to0}m_A(r,t)=\lim_{r\to\pi}m_A(r,t)=0$.
\end{proof}

\begin{thm}[Energy inequality]\label{thmdecompinequality}
Let $a\geq0$ and $\V=(\V^-,\V^+)$ with $\V^-=\sum_{\ell=1}^{\infty}\lambda_{\ell}^-\V_{\ell}^-\in Y^-$ and $\V^+=\sum_{\ell=1}^{\infty}\lambda_{\ell}^+\V_{\ell}^+\in Y^+$ be an $a$-admissible perturbation (cf.~Definition \ref{def:regperturbation}). Assume that the coefficients $\lambda_{\ell}^-$ are non-negative, and the coefficients $\lambda_{\ell}^+$ are non-positive. Then for every $A\in\A(P)$ there holds the inequality 
\begin{eqnarray}\label{decompinequality}
\mathcal{YM}^{\V^-}(A)\geq\frac{\pi}{4}\mathcal E^{\V^+}(x_A),
\end{eqnarray}
with equality if $A$ is Yang--Mills. In this case, the loop $x_A$ is a geodesic.
\end{thm}

\begin{proof}
In view of Remarks \ref{rem:perturbationsYMnonneg} and \ref{rem:perturbationsEnonneg} and the assumptions on $\lambda_{\ell}^{\pm}$ it follows that $\V^-\geq0$ and $\V^+\leq0$. Hence it suffices to prove the inequality in the case of vanishing perturbations $\V^{\pm}=0$, where it follows from the energy identity \eqref{decompequation} (note that the last term in \eqref{decompequation} is non-negative). Equality in the case where $A$ is a Yang--Mills connection and the assertion that then $x_A$ is a closed geodesic follow from the discussion in \cite[p.~236]{FriedrichHabermann}.
\end{proof}

The following is an infinitesimal version of the energy identity \eqref{decompequation}.

\begin{prop}\label{prop:infindecompequation}
Let $A\in\mathcal C^-$ and $x=\Phi(A)\in\mathcal C^+$.  Let $\alpha\in\Omega^1(\Sigma,\ad(P))$ such that $\langle\alpha,\He_A\alpha\rangle<0$ and $\beta\coloneqq\di\Phi(A)\alpha$. Then it follows that $\langle\beta,H_x\beta\rangle<0$.
\end{prop}

\begin{proof}
The claim is an immediate consequence of the energy identity \eqref{decompequation} which implies that for sufficiently small $\eps>0$ the map $\eps\mapsto\E(x_{A+\eps\alpha})$ is strictly monotone decreasing. 
\end{proof}

\section{A priori estimates}\label{sect:aprioriestimates}

\subsubsection*{Estimates involving $\mathcal D_A^{\delta}$} 

For $A\in\A(P)$ we let $\mathcal H_A$ denote the augmented Yang--Mills Hessian as in \eqref{def:augmentedYMH}. Unless otherwise stated we assume in the following that the term $\V^-$ in $\mathcal H_A$ vanishes.

\begin{lem}\label{continuityproperiesDA}
For every $A^-\in\A(P)$ there exists $c(A^-)>0$ such that the operator $\mathcal D_A$ in \eqref{def:operatorDA} satisfies the estimate
\begin{eqnarray*}
\|\mathcal D_A-\mathcal D_{A^-}\|_{\mathcal L(\Z^{\delta,p,-},\mathcal L^{\delta,p,-})}\leq c(A^-)\|\alpha\|_{C^0(\R^-,C^1(\Sigma))}
\end{eqnarray*}
for all time-dependent connections $A=A^-+\alpha$. Similar estimates hold for the operator $\mathcal D_A^{\delta}$ in \eqref{eq:DAdelta} and for domains with $\R^-$ being replaced by some interval $I\subseteq\R^-$.
\end{lem}

\begin{proof}
Consider the upper left entry in $\mathcal H_A-\mathcal H_{A^-}$, i.e.~
\begin{eqnarray*}
-\ast[\alpha\wedge\ast d_{A^-+\alpha}\,\cdot\,]+d_{A^-}^{\ast}[\alpha\wedge\,\cdot\,]+\ast[\ast(d_{A^-}\alpha+\frac{1}{2}[\alpha\wedge\alpha]\wedge\,\cdot\,].
\end{eqnarray*}
It clearly admits an estimate as claimed. The other terms follow similarly.
\end{proof}

The following is a basic estimate involving the operator $\mathcal D_A^{\delta}$ in \eqref{eq:DAdelta}, here for a stationary path $A(s)\equiv A$ (for all $s\in\R^-$).
 
\begin{lem}\label{DAsemiFredholm}
Let $A\in\A(P)$ and assume for some $\delta>0$ that the operator $\mathcal H_A+\delta$ is injective. Then for every $p\geq2$ there exists a constant $c(A,p,\delta)$ and a compact operator $R\colon\Z^{p,-}\rightarrow\mathcal L^{p,-}$ such that the estimate
\begin{eqnarray}\label{coercivityD_A}
\|(\alpha,\psi)\|_ {\Z^{p,-}}\leq c(A,p,\delta)\big(\|\mathcal D_A^{\delta}(\alpha,\psi)\|_{\mathcal L^{p,-}}+\|R(\alpha,\psi)\|_{\mathcal L^{p,-}}\big)
\end{eqnarray}
holds for all $(\alpha,\psi)\in\Z^{p,-}$. Moreover, the operator $\mathcal D_A^{\delta}\colon\Z^{p,-}\rightarrow\mathcal L^{p,-}$ is surjective and has finite-dimensional kernel of dimension 
\begin{eqnarray*}
\dim\ker\mathcal D_A^{\delta}=\ind A.
\end{eqnarray*}
(The integer $\ind A$ denoting the number of negative eigenvalues of $\mathcal H_A+\delta$).
\end{lem}

\begin{proof}
The proof follows the lines of \cite[Theorem 8.5]{Web} and consists of four steps. Throughout we set $\zeta\coloneqq(\alpha,\psi)$ and denote $H\coloneqq L^2(\Sigma,T^{\ast}\Sigma\otimes\ad(P))\oplus L^2(\Sigma,\ad(P))$. 
\setcounter{step}{0}
\begin{step}
The statement on surjectivity and the kernel is true in the case $p=2$.
\end{step}
The operator $\mathcal H_A+\delta$ with domain as in \eqref{def:domainHA} is an unbounded self-adjoint operator on the Hilbert space $H$, cf.~\cite[Proposition 5.1]{Swoboda1}. Denote by $E^-$ and $E^+$ its negative, respectively positive eigenspaces. Since $\mathcal H_A+\delta$ is assumed to be injective, $H$ splits as an orthogonal sum $H=E^-\oplus E^+$. Let $P^{\pm}$ denote the projections onto $E^{\pm}$ and set $\mathcal H^{\pm}=(\mathcal H_A+\delta)|_{E^{\pm}}$. As $\mathcal H^-$ and $-\mathcal H^+$ are negative-definite self-adjoint operators, it follows from the Hille--Yosida theorem (cf.~for instance \cite[Section X.8]{ReSi}) that they generate strongly continuous contraction semigroups $s\mapsto e^{s\mathcal H^-}$ on $E^-$ respectively $s\mapsto e^{-s\mathcal H^+}$ on $E^+$, both defined for $s\geq0$. This allows us to define the map $K\colon\mathbbm R\to\mathcal L(H)$ by
\begin{eqnarray}\label{eq:kernel}
K(s)\coloneqq\begin{cases}
-e^{-s\mathcal H^-}P^-&\textrm{for}\;s\leq0,\\
e^{-s\mathcal H^+}P^+&\textrm{for}\;s>0.
\end{cases}
\end{eqnarray}
As one easily checks, $K$ is strongly continuous in $\mathbb R\setminus\{0\}$ and its pointwise operator norm satisfies
\begin{eqnarray*} 
\|K(s)\|_{\mathcal L(H)}\leq e^{-\delta_0|s|}
\end{eqnarray*}
for $\delta_0>0$ the smallest (in absolute value) eigenvalue of $\mathcal H_A+\delta$. Now consider the operator $Q\colon\mathcal L^{2,-}\to\Z^{2,-}$ defined by
\begin{eqnarray}\label{eq:defopQ}
(Q\eta)(s)\coloneqq\int_{-\infty}^0K(s-\sigma)\eta(\sigma)\,d\sigma.
\end{eqnarray}
It satisfies
\begin{eqnarray*}
\frac{d}{ds}(Q\eta)(s)&=&\frac{d}{ds}\int_{-\infty}^se^{-(s-\sigma)\mathcal H^+}P^+\eta(\sigma)\,d\sigma-\frac{d}{ds}\int_{s}^0e^{-(s-\sigma)\mathcal H^-}P^-\eta(\sigma)\,d\sigma\\
&=&P^+\eta(s)-\int_{-\infty}^s\mathcal H^+e^{-(s-\sigma)\mathcal H^+}P^+\eta(\sigma)\,d\sigma\\
&&+P^-\eta(s)+\int_{s}^0\mathcal H^-e^{-(s-\sigma)\mathcal H^-}P^-\eta(\sigma)\,d\sigma\\
&=&\eta(s)-(\mathcal H_A+\delta)(Q\eta)(s).
\end{eqnarray*}
From this calculation we see that
\begin{eqnarray*}
\mathcal D_A^{\delta}Q\eta= \frac{d}{ds}(Q\eta)+(\mathcal H_A+\delta)(Q\eta)=\eta,
\end{eqnarray*}
so $Q$ is a right-inverse of $\mathcal D_A^{\delta}$. This proves surjectivity of $\mathcal D_A^{\delta}$. Now let $\zeta\in\ker\mathcal D_A^{\delta}$ and assume that $(\mathcal H_A+\delta)\zeta(0)=\lambda\zeta(0)$  for some $\lambda\in\mathbbm R$. Then $\zeta(s)=e^{-\lambda s}\zeta(0)$, which is contained in $\mathcal L^{2,-}$ if and only if $\lambda<0$. Therefore $\zeta\in\Z^{2,-}$ satisfies $\mathcal D_A^{\delta}\zeta=0$ if and only if $\zeta(0)\in\mathcal H^-$. This shows that $\mathcal H^-$ and $\ker\mathcal D_A^{\delta}$ are isomorphic to each other.

\begin{step}
For every $p\geq2$ there exists  a constant $c_1(A,p)$ such that the following holds. If $\zeta\in\Z^{2,-}$ and $\mathcal D_A^{\delta}\zeta\in\mathcal L^{p,-}$, then $\zeta\in\Z^{p,-}$ and
\begin{eqnarray*} 
\|\zeta\|_{\Z^{p,-}}\leq c_1(A,p)\big(\|\mathcal D_A^{\delta}\zeta\|_{\mathcal L^{p,-}}+\|\zeta\|_{L^p(\R^-,H)}\big).
\end{eqnarray*}
\end{step}
The claim follows from standard arguments based on the linear estimate 
\begin{eqnarray}\label{apriori2A}
\|\zeta\|_ {\Z^{p,-}([-1,0])}\leq c(A,p)\big(\|\mathcal D_A^{\delta}\zeta\|_{L^p([-2,0])}+\|\zeta\|_{L^p([-2,0])}\big),
\end{eqnarray}
cf.~\cite[Proposition A.6]{Swoboda}. Full details are given in \cite[Lemma 3.20]{Swoboda}.

\begin{step}
The operator $Q\colon\mathcal L^{p,-}\to L^p(\R^-,H)$ defined through \eqref{eq:defopQ} is bounded, for every $p\geq2$. (In the following, we let $c_2(A,p)$ denote its operator norm.)
\end{step}

The claim follows from Young's convolution inequality with
\begin{align}\label{eq:Young}
\nonumber\|Q\eta\|_{L^p(\R^-,H)}=&\|K\ast\eta\|_{L^p(\R^-,H)}\\
\leq&\|K\|_{L^1(\R^-,\mathcal L(H))}\|\eta\|_{L^p(\R^-,H)}\leq\frac{1}{\delta_0}\|\eta\|_{\mathcal L^{p,-}}.
\end{align}
In the last step we used that for $p\geq2$ the $\mathcal L^{p,-}$ norm dominates the $L^p(\R^-,H)$ norm.

\begin{step}
We prove the lemma.
\end{step}
The estimates of Step 2 and Step 3 imply that
\begin{eqnarray*}
\lefteqn{\|\zeta\|_{\Z^{p,-}}\leq c_1(A,p)\big(\|\mathcal D_A^{\delta}\zeta\|_{\mathcal L^{p,-}}+\|\zeta\|_{L^p(\R^-,H)}\big)}\\
&\leq&c_1(A,p)\big(\|\mathcal D_A^{\delta}\zeta\|_{\mathcal L^{p,-}}+\|Q\mathcal D_A^{\delta}\zeta\|_{L^p(\R^-,H)}+\|\zeta-Q\mathcal D_A^{\delta}\zeta\|_{L^p(\R^-,H)}\big)\\
&\leq&c_1(A,p)\big((1+c_2(A,p))\|\mathcal D_A^{\delta}\zeta\|_{\mathcal L^{p,-}}+\|\zeta-Q\mathcal D_A^{\delta}\zeta\|_{L^p(\R^-,H)}\big).
\end{eqnarray*}
This shows \eqref{coercivityD_A} because the operator 
\begin{eqnarray*}
R\coloneqq\mathbbm 1-Q\mathcal D_A^{\delta}\colon\Z^{p,-}\to L^p(\R^-,H)
\end{eqnarray*}
has finite rank (of dimension equal to $\dim\ker\mathcal D_A^{\delta}$) and therefore is compact. To prove surjectivity we note that the operator $\mathcal D_A^{\delta}$ has closed range by \eqref{coercivityD_A} and the abstract closed range lemma (cf.~\cite[p.~14]{Sal2}). Hence it suffices to show that $\ran\mathcal D_A^{\delta}$ is dense in $\mathcal L^{p,-}$. Let $\eta\in\mathcal L^{p,-}\cap\mathcal L^{2,-}$ be given. The latter is a dense subspace of $\mathcal L^{p,-}$ because it contains all compactly supported smooth functions. By Step 1 there exists some $\zeta\in\Z^{2,-}$ such that $\mathcal D_A\xi=\eta$. From \eqref{coercivityD_A} and the assumption $\eta\in\mathcal L^{p,-}$ it follows that $\xi\in\Z^{p,-}$ which implies surjectivity. Again \eqref{coercivityD_A} together with finiteness of the rank of $R$ shows that the kernels of the operators $\mathcal D_A^{\delta}\colon\Z^{2,-}\to\mathcal L^{2,-}$ and $\mathcal D_A^{\delta}\colon\Z^{p,-}\to\mathcal L^{p,-}$ coincide, for all $p\geq2$. Hence the assertion on the kernel follows from Step 1. This finishes the proof of the lemma.
\end{proof}

\begin{lem}\label{DAsurjective} 
Let $s\mapsto A(s)$, $s\in\R^-$, be a smooth solution of \eqref{EYF} such that for a Yang--Mills connection $A^-$ the asymptotic condition $\lim_{s\to-\infty}A(s)=A^-$ is satisfied in $C^1(\Sigma)$. Let $\delta>0$ be such that the operator $\mathcal H_A+\delta$ is injective. Then the operator $\mathcal D_A^{\delta}\colon\Z^{p,-}\to\mathcal L^{p,-}$ associated with $A$ is surjective and has finite-dimensional kernel of dimension 
\begin{eqnarray}\label{eq:formulaDAsurj}
\dim\ker\mathcal D_A^{\delta}=\ind A^-.
\end{eqnarray}
(The integer $\ind A$ denoting the number of negative eigenvalues of $\mathcal H_A+\delta$).
\end{lem}

\begin{proof}
The proof is divided into three steps.
\setcounter{step}{0}
\begin{step}
Stationary case.
\end{step}
Consider the case where the path $A\equiv A^-$ is stationary. In this case Lemma \ref{DAsemiFredholm} applies and yields surjectivity of $\mathcal D_A^{\delta}$ and formula \eqref{eq:formulaDAsurj}. In particular, $\mathcal D_A^{\delta}$ is a Fredholm operator.

\begin{step}
Nearby case.
\end{step} 
Let us assume that for some sufficiently small $\eps>0$ the condition 
\begin{eqnarray}\label{assumptionnearby}
\|A-A^-\|_{\mathcal C^0(\mathbbm R^-,\mathcal C^1(\Sigma))}<\eps
\end{eqnarray}
is satisfied. We here consider $A^-$ as a stationary connection over $\mathbbm R^-\times\Sigma$. Surjectivity and the Fredholm index are preserved under small perturbations with respect to the operator norm. By Lemma \ref{continuityproperiesDA}, the operator norm of $\mathcal D_A^{\delta}$ depends continuously on $A$ with respect to the $C^0(\mathbbm R^-,C^1(\Sigma))$ topology. Therefore surjectivity is implied by assumption \eqref{assumptionnearby}. As the Fredholm indices of $\mathcal D_{A^-}^{\delta}$ and $\mathcal D_A^{\delta}$ coincide it follows that \eqref{eq:formulaDAsurj} holds true in the nearby case.

\begin{step}
General case.
\end{step}
The general case can be reduced to the nearby case by a standard argument as e.g.~carried out in the proof of \cite[Proposition 8.3]{Web}.
\end{proof}

\subsubsection*{Estimates involving $\mathcal D_x^{\delta}$}

The following lemma gives a basic estimate for the operator $\mathcal D_x^{\delta}$ in \eqref{eq:Dgdelta}.

\begin{lem}\label{lem:linestimateDg}
Let $p\geq2$ and $x\in C^{\infty}(\R^1\times S^1,G)$ such that $\lim_{s\to\infty}x(s)=x^+$ holds in $C^1(S^1)$ for some $x^+\in C^{\infty}(S^1,G)$. Then for every sufficiently small $\delta>0$ there exist positive constants $c$ and $T$, which depend only on $x$, $p$, and $\delta$, such that for every $\xi\in\Z^{p,+}$ the estimate
\begin{eqnarray}\label{eq:parabolicinitialvalue}
\|\xi\|_{\Z^{p,+}}\leq c\big(\|\mathcal D_x^{\delta}\xi\|_{\mathcal L^{p,+}}+\|\xi\|_{L^p([0,T],L^2(S^1))}+\|\xi(0)\|_{L^2(S^1)}\big)
\end{eqnarray} 
is satisfied.
\end{lem}

\begin{proof}
We start with the standard parabolic estimate
\begin{eqnarray}\label{eq:parabolicinitialvalue1}
\|\xi\|_{\Z^{p,+}}\leq c(p,x)\big(\|\mathcal D_x^{\delta}\xi\|_{\mathcal L^{p,+}}+\|\xi\|_{L^p(\R^+,L^2(S^1))}\big)
\end{eqnarray}
as obtained (for any $p\geq2$) in Step 3 of the proof of \cite[Theorem 8.5]{Web}. To prove the lemma, it therefore remains to further estimate the last term in \eqref{eq:parabolicinitialvalue1}, which we split into integrals over $[0,T]$ and $[T,\infty)$ for sufficiently large $T>0$. We apply Lemma \ref{lem:linL2est2} to the self-adjoint operator $L\coloneqq H_{x^+}-\delta$ on the Hilbert space $L^2(S^1)$, which has spectrum bounded away from $0$ for sufficiently small $\delta>0$. Defining $\mathcal D_{x^+}^{\delta}\coloneqq\frac{d}{ds}+L$ this yields for a constant $c(x^+,p)$ the estimate
\begin{eqnarray}\label{eq:estimateLpL2}
\|\xi\|_{L^p([T,\infty),L^2(S^1))}\leq c(x^+,p)\big(\|\mathcal D_{x^+}^{\delta}\xi\|_{L^p([T,\infty),L^2(S^1))}+\|\xi(T)\|_{L^2(S^1)}\big),
\end{eqnarray}
for all $\xi\in L^p([T,\infty),L^2(S^1))$. Denoting by $\mathcal L$ the space of bounded linear maps $W^{1,2}(S^1)\to L^2(S^1)$ and by $\|\cdot\|_{\mathcal L}$ the corresponding operator norm it follows for $s\geq T$ the estimate
\begin{align*}
\|\mathcal D_{x^+}^{\delta}\xi(s)-\mathcal D_{x(s)}^{\delta}\xi(s)\|_{L^2(S^1)}=&\|H_{x^+}\xi(s)-H_{x(s)}\xi(s)\|_{L^2(S^1)}\\
\leq&\|H_{x^+}-H_{x(s)}\|_{\mathcal L}\|\xi(s)\|_{W^{1,2}(S^1)}.
\end{align*}
Thus we can further estimate the term $\mathcal D_{x^+}^{\delta}\xi$ in \eqref{eq:estimateLpL2} as
\begin{eqnarray*}
\lefteqn{\|\mathcal D_{x^+}^{\delta}\xi\|_{L^p([T,\infty),L^2(S^1))}}\\
&\leq&\|\mathcal D_{x^+}^{\delta}\xi-\mathcal D_x^{\delta}\xi\|_{L^p([T,\infty),L^2(S^1))}+\|\mathcal D_x^{\delta}\xi\|_{L^p([T,\infty),L^2(S^1))}\\
&\leq&\|H_{x^+}-H_x\|_{L^{\infty}([T,\infty),\mathcal L)}\|\xi\|_{L^p([T,\infty),W^{1,2}(S^1))}+\|\mathcal D_x^{\delta}\xi\|_{L^p([T,\infty),L^2(S^1))}.
\end{eqnarray*}
Using the assumption $\lim_{s\to\infty}x(s)=x^+$ in $C^1(S^1)$, it can be checked that $\|H_{x^+}-H_x\|_{L^{\infty}([T,\infty),\mathcal L)}\to0$ as $T\to\infty$. Hence the term involving $H_{x^+}-H_x$ can be absorbed in the left-hand side of \eqref{eq:parabolicinitialvalue} for $T=T(x)$ sufficiently large. Here we use that the norm of $\Z^{p,+}$ dominates that of $L^p(\R^+,W^{1,2}(S^1))$. Furthermore, the term $\|\mathcal D_x^{\delta}\xi\|_{L^p([T,\infty),L^2(S^1))}$ is controlled by $\|\mathcal D_x^{\delta}\xi\|_{\mathcal L^{p,+}}$ as is clear from the assumption $p\geq2$. The desired estimate now follows after applying Lemma \ref{lem:finintervalest} (with $L(s)\coloneqq H_{x(s)}-\delta$ and $\eps>0$ sufficienty small) to the remaining term $\|\xi(T)\|_{L^2(S^1)}$ in \eqref{eq:estimateLpL2}. This introduces a further term $\eps\|H_x\xi\|_{L^2([0,T],L^2(S^1))}$ which can be absorbed in the left-hand side of the asserted inequality \eqref{eq:parabolicinitialvalue}, and a term $\|\mathcal D_x^{\delta}\xi\|_{L^2([0,T],L^2(S^1))}+\|\xi\|_{L^2([0,T],L^2(S^1))}$ which for $p\geq2$ is dominated by the term $\|\mathcal D_x^{\delta}\xi\|_{L^p([0,T],L^p(S^1))}+\|\xi\|_{L^p([0,T],L^2(S^1))}$ appearing on the right hand side of \eqref{eq:parabolicinitialvalue}.
\end{proof}

\begin{lem}\label{Dgsurjective} 
Let $s\mapsto x(s)$, $s\in\R^+$, be a smooth solution of \eqref{introdloopgradient1} such that for a closed geodesic $x^+$ the asymptotic condition $\lim_{s\to\infty}x(s)=x^+$ is satisfied in $C^1(S^1)$. Then the operator $\mathcal D_x^{\delta}\colon\Z^{p,+}\to\mathcal L^{p,+}$ associated with $x$ is surjective and has finite-dimensional cokernel of dimension
\begin{eqnarray*} 
\dim\coker\mathcal D_x^{\delta}=\ind x^+.
\end{eqnarray*}
(The integer $\ind x^+$ denoting the number of negative eigenvalues of $H_{x^+}-\delta$).
\end{lem}

\begin{proof}
The proof of an analogous result in \cite[Proposition 8.3]{Web} for the backward halfcylinder $\mathbbm R^-\times S^1$ carries over to the present situation by taking the adjoint of $\mathcal D_x^{\delta}$ and time-reversal $s\mapsto-s$.
\end{proof}

\subsubsection*{Further linear estimates}

\begin{lem}\label{lem:linL2est2}
Let $H$ be a real Hilbert space and $L\colon\dom(L)\to H$ be the infinitesimal generator of a strongly continuous one-parameter semigroup on $H$. We assume that the spectrum of $L$ is contained in $(-\infty,-\lambda]\cup[\lambda,\infty)$ for some $\lambda>0$. Let $p\geq1$ and $\eta\in L^p(\R^+,H)$. Then any solution $\xi\colon\R^+\to H$ of the equation $\dot\xi+L\xi=\eta$ satisfies the estimate
\begin{eqnarray*}
\|\xi\|_{L^p(\R^+,H)}\leq\frac{2}{\lambda}\|\eta\|_{L^p(\R^+,H)}+\frac{1}{\sqrt[p]{\lambda p}}\|\xi_0\|_H.
\end{eqnarray*}
Here we denote $\xi_0\coloneqq\xi(0)$.
\end{lem}

\begin{proof}
Similar to the proof of Lemma we can construct a right-inverse $Q$ to the operator $\frac{d}{ds}+L$ by convolution with a kernel as in \eqref{eq:kernel}. Then $\|Q\eta\|_{L^p(\R^+,H)}\leq\frac{2}{\lambda}\|\eta\|_{L^p(\R^+,H)}$ as follows by applying Young's inequality as in \eqref{eq:Young}. Furthermore, the $L^p(\R^+,H)$ norm of elements $\xi$ in the kernel of $\frac{d}{ds}+L$ is bounded above by 
\begin{eqnarray*}
\big(\int_0^{\infty}e^{-\lambda ps}\|\xi_0\|_H^p\,ds\big)^{\frac{1}{p}}=\frac{1}{\sqrt[p]{\lambda p}}\|\xi_0\|.
\end{eqnarray*}
The claim then follows.
\end{proof}

\begin{lem}\label{lem:finintervalest}
Let $H$ be a Hilbert space. Assume $\xi\colon[s_0,s_1]\to H$ satisfies the equation $\dot\xi+L\xi=\eta$ for a path $s\mapsto L(s)$ of (densely defined) linear operators on $H$. Then for every $\eps>0$ there holds the estimate
\begin{multline*}
\|\xi(s_1)\|_H^2-\|\xi(s_0)\|_H^2\\
\leq(1+\eps^{-1})\|\xi\|_{L^2([s_0,s_1],H)}^2+\eps\|L\xi\|_{L^2([s_0,s_1],H)}^2+\|\eta\|_{L^2([s_0,s_1],H)}^2.
\end{multline*}
\end{lem}

\begin{proof}
We integrate the equation
\begin{eqnarray*}
\frac{d}{ds}\frac{1}{2}\|\xi(s)\|_H^2=\langle\xi(s),\eta(s)\rangle-\langle\xi(s),L(s)\xi(s)\rangle
\end{eqnarray*}
over the interval $[s_0,s_1]$ and apply to the two terms on the right-hand side the Cauchy-Schwarz inequality. This immediately yields the result.
\end{proof}

The following is a general interpolation lemma for operators of type $D=\frac{d}{ds}+L(s)$.
  
\begin{lem}[Interpolation Lemma]\label{interpolationlemma}
Let $V\subseteq H\subseteq V^{\ast}$ be a Gelfand triple. Assume that the family $L(s)\colon V\rightarrow V^{\ast}$ ($s\in\mathbbm R^-$) of operators satisfies for all $\xi\in V$ the uniform bound
\begin{eqnarray*}
\|\xi\|_V^2\leq c_1\langle L(s)\xi,\xi\rangle+c_2\|\xi\|_H^2
\end{eqnarray*}
for constants $c_1,c_2>0$. Then for any $\eps>0$ and $p\geq2$ there holds the estimate 
\begin{multline*}
\frac{1}{p}\|\xi(0)\|_H^p+\big(\frac{1}{c_1}-\frac{1}{2\eps}\big)\int_{-\infty}^0\|\xi(s)\|_H^{p-2}\|\xi(s)\|_V^2\,ds\\
\leq\big(\frac{\eps(p-2)}{2p}+\frac{c_2}{c_1}\big)\|\xi\|_{L^p(\R^-,H)}^p+\frac{\eps}{p}\|\dot\xi+L\xi\|_{L^p(\R^-,V^{\ast})}^p
\end{multline*}
for all $\xi\in W^{1,p}(\mathbbm R^-,V^{\ast})\cap L^p(\mathbbm R^-,V)$.
\end{lem}
 
\begin{proof}
Set $\eta\coloneqq\dot\xi+L\xi\in L^p(\mathbbm R^-,V^{\ast})$. For every $s\in\mathbbm R^-$ there holds the estimate 
\begin{eqnarray*}
\lefteqn{\frac{1}{p}\frac{d}{ds}\|\xi(s)\|_H^p=\|\xi(s)\|_H^{p-2}\langle\dot\xi(s),\xi(s)\rangle_H}\\
&=&\|\xi(s)\|_H^{p-2}\langle\eta(s)-L(s)\xi(s),\xi(s)\rangle_H\\
&\leq&\|\xi(s)\|_H^{p-2}\big(\|\eta(s)\|_{V^{\ast}}\|\xi(s)\|_V-\frac{1}{c_1}\|\xi(s)\|_V^2+\frac{c_2}{c_1}\|\xi(s)\|_H^2\big)\\
&\leq&\big(\frac{1}{2\eps}-\frac{1}{c_1}\big)\|\xi(s)\|_H^{p-2}\|\xi(s)\|_V^2+\frac{\eps}{2}\|\xi(s)\|_H^{p-2}\|\eta(s)\|_{V^{\ast}}^2+\frac{c_2}{c_1}\|\xi(s)\|_H^p,
\end{eqnarray*}
for any constant $\eps>0$. Integrating this inequality over $\mathbbm R^-$ and applying H\"older's inequality yields
\begin{eqnarray*}
\lefteqn{\frac{1}{p}\|\xi(0)\|_H^p=\int_{-\infty}^0\frac{1}{p}\frac{d}{ds}\|\xi(s)\|_H^p\,ds}\\
&\leq&\int_{-\infty}^0\big(\frac{1}{2\eps}-\frac{1}{c_1}\big)\|\xi(s)\|_H^{p-2}\|\xi(s)\|_V^2+\frac{\eps}{2}\|\xi(s)\|_H^{p-2}\|\eta(s)\|_{V^{\ast}}^2+\frac{c_2}{c_1}\|\xi(s)\|_H^p\,ds\\
&\leq&\int_{-\infty}^0\big(\frac{1}{2\eps}-\frac{1}{c_1}\big)\|\xi(s)\|_H^{p-2}\|\xi(s)\|_V^2\,ds+\frac{c_2}{c_1}\int_{-\infty}^0\|\xi(s)\|_H^p\,ds\\
&&+\frac{\eps}{2}\Big(\int_{-\infty}^0\|\xi(s)\|_H^p\,ds\Big)^{\frac{p-2}{p}}\cdot\Big(\int_{-\infty}^0\|\eta(s)\|_{V^{\ast}}^p\,ds\Big)^{\frac{2}{p}}.
\end{eqnarray*}
We now apply Young's inequality to the product term in the last line. The claim then follows.  
\end{proof}

\end{document}